\documentclass[times]{oupau}

\usepackage{float}
\usepackage{xcolor}
\usepackage{mathbbol}
\usepackage{tikz}
\usetikzlibrary{matrix,arrows,cd,calc}
\usepackage{thm-restate}
\usepackage{mathtools}
\usepackage{amsmath}
\usepackage{amsthm}
\usepackage{amssymb}
\usepackage{mathbbol}
\usepackage[hidelinks]{hyperref}
\usepackage{xcolor}
\definecolor{darkgreen}{rgb}{0,0.45,0}
\hypersetup{colorlinks,urlcolor=blue,citecolor=darkgreen,linkcolor=darkgreen,linktocpage}
\usepackage[capitalize]{cleveref}
\usepackage{xspace}
\usepackage{caption}
\usepackage{subcaption}
\usepackage{enumerate}
\usepackage{centernot}
\usepackage{textcomp}
\usepackage[normalem]{ulem}
\usepackage[style=alphabetic,maxalphanames=30,maxbibnames=99]{biblatex}
\theoremstyle{definition}

\newtheorem{lem}[theorem]{Lemma}
\newtheorem{cor}[theorem]{Corollary}
\newtheorem{rem}[theorem]{Remark}

\newtheorem{dfn}[theorem]{Definition}
\newtheorem{exm}[theorem]{Example}

\crefname{thm}{Theorem}{Theorems}
\crefname{lem}{Lemma}{Lemmas}
\crefname{cor}{Corollary}{Corollaries}

\newtheorem*{proposition*}{Proposition}
\theoremstyle{definition}
\newtheorem{notation}[theorem]{Notation}

\def\noteson{\gdef\luis##1{\noindent{\color{blue}[Luis: ##1]}}
\gdef\uli##1{\noindent{\color{violet}[Uli: ##1]}}
\gdef\todo##1{\noindent{\color{red}[todo: ##1]}}}

\noteson

\newcommand{\define}[1]{\emph{#1}}

\renewcommand{\to}{\xrightarrow{\;\;\;}}

\renewcommand{\mod}{\mathrm{mod}\xspace}

\renewcommand{\epsilon}{\varepsilon}
\renewcommand{\phi}{\varphi}

\DeclareMathOperator{\End}{End}

\DeclareMathOperator{\colim}{colim}

\newcommand{\rk}{\mathsf{rk}}

\newcommand{\vect}{\mathbf{vec}}

\usepackage{ifthen}

\newcommand{\stkout}[1]{\ifmmode\text{\sout{\ensuremath{#1}}}\else\sout{#1}\fi}

\DeclareMathAlphabet{\mathpzc}{OT1}{pzc}{m}{it}

\usepackage[mathscr]{euscript}

\makeatletter
\newcommand\DEFINEALPHABETLOOP[3]{%
  \ifx\relax#3\expandafter\@gobble\else\expandafter\@firstofone\fi
  {\expandafter\newcommand\expandafter*\csname#3#1\endcsname{#2{#3}}%
   \DEFINEALPHABETLOOP{#1}{#2}}%
}%
\newcommand\Definealphabet[2]{%
  \DEFINEALPHABETLOOP{#1}{#2}abcdefghijklmnopqrstuvwxyzABCDEFGHIJKLMNOPQRSTUVWXYZ\relax
}%
\makeatother
\Definealphabet{bb}{\mathbb}
\Definealphabet{cal}{\mathcal}
\Definealphabet{bf}{\mathbf}
\Definealphabet{sf}{\mathsf}
\Definealphabet{rm}{\mathrm}
\Definealphabet{frak}{\mathfrak}
\Definealphabet{scr}{\mathscr}
\renewbibmacro*{doi+eprint+url}{%
  \iftoggle{bbx:doi}
    {\printfield{doi}}
    {}%
  \newunit\newblock
  \ifboolexpr{togl {bbx:eprint} and test {\iffieldundef{doi}}}
    {\usebibmacro{eprint}}
    {}%
  \newunit\newblock
  \ifboolexpr{togl {bbx:url} and test {\iffieldundef{doi}}  and test {\iffieldundef{eprint}}}
    {\usebibmacro{url+urldate}}
    {}}

\begin{document}

\title{Multi-Parameter Persistence Modules are Generically Indecomposable}

\shorttitle{Multi-Parameter Persistence Modules are Generically Indecomposable}

\volumeyear{-}
\paperID{-}

\author{Ulrich Bauer\affil{1} and Luis Scoccola\affil{2}}
\abbrevauthor{U. Bauer and L. Scoccola}
\headabbrevauthor{Bauer, U., and L. Scoccola}


\address{%
\affilnum{1}Department of Mathematics, TUM School of Computation, Information and Technology, Technical University of Munich, Germany (\url{ulrich.bauer@tum.de})
and
\affilnum{2}Mathematical Institute, University of Oxford, United Kingdom (\url{luis.scoccola@maths.ox.ac.uk})}

\correspdetails{luis.scoccola@maths.ox.ac.uk}

\received{-}
\revised{-}
\accepted{-}

\communicated{-}

\begin{abstract}%
Algebraic persistence studies persistence modules (typically, linear representations of the poset $\mathbf{R}^n$ with $n \geq 1$) and
the algebraic relationships between persistence modules that are interleaved.
The notion of $\epsilon$-interleaving between persistence modules is a generalization of the notion of isomorphism (recovering isomorphism when $\epsilon = 0$), which can be used to quantify how far any two persistence modules are from being isomorphic.
An emblematic example of this kind of study is the algebraic stability theorem, which strengthens the Krull--Schmidt property of one-parameter persistence modules (representations of $\Rbf$) by generalizing isomorphism to interleaving:
If a pair of one-parameter persistence modules is $\epsilon$-interleaved, then there exists a partial matching between the indecomposable summands of the two modules such that matched indecomposables are $\epsilon$-interleaved and unmatched indecomposables are $\epsilon$-interleaved with the zero module.
Our first main result implies that the obvious extension of the algebraic stability theorem to the case of multi-parameter persistence modules (representations of $\Rbf^n$ with $n \geq 2$) fails spectacularly:
Any finitely presentable multi-parameter persistence module can be approximated arbitrarily well by an indecomposable module.
Our second main result states that modules that are sufficiently close to an indecomposable decompose as a direct sum of an indecomposable and a nearly trivial module.
We derive from these two results several consequences about the interplay between the algebraic and the topological properties of multi-parameter persistence modules.
These results provide strong motivation for approaching multi-parameter persistence in a way that does not rely on directly decomposing modules by indecomposables.
\end{abstract}


\maketitle

\section{Introduction}
\label{section:introduction}

The \emph{decomposition theorem} for one-parameter persistence modules states that any finitely presentable linear representation of a totally ordered set decomposes uniquely as a direct sum of \emph{interval modules} \cite{gabriel,webb,zomorodian-carlsson,crawley-boevey}.
Interval modules owe their name to the fact that they are completely characterized by their support, which is necessarily an interval of the totally ordered set.
This theorem is of fundamental importance in topological data analysis and inference, where indecomposables with large support (long intervals) are interpreted as prominent topological features of data, corresponding to large-scale features of the process that generated the data, while indecomposables with small support (short intervals) are understood as features caused by noise, randomness, or local geometry of the underlying process \cite{
edelsbrunner-letscher-zomorodian,
edelsbrunner-harer,
chazal-guibas-oudot-skraba,
hiraoka-shirai-trinh,
bubenik-hull-patel-whittle,
bobrowski-skraba,
}.
The consistency of this methodology is guaranteed by various stability theorems.
In particular, the \emph{algebraic stability theorem} \cite{
chazal-cohen-steiner-glisse-guibas-oudot,
chazal-silva-glisse-oudot,
bauer-lesnick} states that the decomposition of one-parameter persistence modules into their indecomposable summands is a $1$-Lipschitz operation, when persistence modules are compared using the \emph{interleaving distance} \cite{chazal-cohen-steiner-glisse-guibas-oudot,lesnick}, an algebraically-defined distance satisfying a desirable universal property.
While the algebraic stability theorem was originally motivated by scientific applications, it has recently found applications in areas of pure mathematics such as analysis and symplectic geometry \cite{MR4249570,MR4413744,buhovsky2022coarse}.

Multi-parameter persistence modules
arise naturally 
as the homology of topological spaces filtered by more than one real parameter
%
\cite{
    carlsson-zomorodian,
    biasotti-cerri-frosini-giorgi-landi,
    lesnick-wright,
    escolar-hiraoka,
    mcinnes-healy,
    vipond,
    kim-memoli-2,
    botnan-lesnick-2,
    }.
When moving from one to multiple parameters, the approach to understanding an arbitrary persistence module by completely characterizing its indecomposable summands is thought to be infeasible as it involves classifying all indecomposable representations of all finite dimensional algebras, in the following sense.

\newcommand{\statementmultiparameterwild}{%
    Let $\Pscr$ be a totally ordered set with at least $4$ elements, and let $n \geq 2$.
    Let~$\kbb$ be any field and let $\Lambda$ be a finite dimensional $\kbb$-algebra.
    Given a set of generators of~$\Lambda$, there is an explicit injective function 
    mapping isomorphism classes of indecomposable, finite dimensional $\Lambda$-modules to isomorphism classes of indecomposable, finitely presentable functors $\Pscr^n \to \vect_\kbb$. 
}

\begin{proposition*}
    \statementmultiparameterwild
\end{proposition*}

The proposition
is standard and follows from well known results in representation theory; for completeness, we give a proof with references in \cref{appendix}.

One can interpret the proposition
as a no-go result stating that, contrary to the one-parameter case, no significant reduction in algebraic complexity is expected from decomposing arbitrary multi-parameter persistence modules into their indecomposable summands.
It is important to keep this in mind, since, in order to devise feasible approaches to multi-parameter persistence, it is necessary to understand the limits of the mathematical tools used to study multi-parameter persistence modules.
However, this no-go result is a purely algebraic one; meanwhile, key results in algebraic persistence -- such as stability theorems -- build on the interplay between the algebraic and the metric properties of persistence modules.
In particular,
the proposition
says nothing about the behavior of the algebraic structure of multi-parameter persistence modules under perturbations in the interleaving distance.
This motivates various questions:
\begin{enumerate}[$(i)$]
    \item With respect to the topology induced by the interleaving distance, how do generic multi-parameter persistence modules decompose?
    \item Which multi-parameter persistence modules are structurally stable, meaning that they have the property that any sufficiently nearby module has a similar decomposition into indecomposables?
    \item Does there exist a simple family of indecomposable multi-parameter persistence modules $\Fcal$ such that any module can be approximated arbitrarily well, in the interleaving distance, by a direct sum of modules in~$\Fcal$?
\end{enumerate}


This paper contains two main theorems
(\cref{theorem:indecomposables-dense} and \cref{proposition:stability-indecomposability})
and four corollaries of these theorems
(\cref{theorem:main-theorem,corollary:structural-stability,theorem:dense-F-decomposables,theorem:dense-F-complicated}), which we state in this introduction.
\cref{theorem:main-theorem,corollary:structural-stability} provide answers to questions $(i)$ and $(ii)$ respectively, while \cref{theorem:dense-F-decomposables,theorem:dense-F-complicated} provide answers to question $(iii)$.
The main conclusion that we draw from the results in this paper is that, contrary to the one-parameter case, the
operation of decomposing arbitrary multi-parameter persistence modules into their indecomposable summands is highly unstable.
As such, our results provide further motivation for various proposed approaches to understanding the structure and stability of multi-parameter persistence modules that do not rely on directly decomposing modules by indecomposables
\cite{
biasotti-cerri-frosini-giorgi-landi,
carlsson-zomorodian,
cerri-difabio-ferri-frosini-landi,
scolamiero-chacholski-lundman-ramanujan-oberg,
miller,
kim-memoli,
botnan-oppermann-oudot,
blanchette-brustle-hanson,
lesnick-wright-2,
asashiba2023approximation,
mccleary2021edit,
botnan-oppermann-oudot-scoccola,
cacholski-guidolin-ren-scolamiero-tombari,
oudot-scoccola,
asashiba-escolar-nakashima-yoshiwaki,
bjerkevik2023stabilizing,
},
and give a sense of how complicated arbitrary indecomposables can be, complementing results such as the ones in \cite{carlsson-zomorodian,buchet-escolar,moore}.



\subparagraph{Contributions}
Let $n \geq 1 \in \Nbb$.
An $n$-parameter persistence module is a functor $\Rbf^n \to \vect$, where $\vect$ is the category of finite dimensional vector spaces over a fixed field.
Note that we only consider persistence modules valued in finite dimensional vector spaces, which are often referred to as pointwise finite dimensional persistence modules.
Recall that the \emph{interleaving distance}, denoted $d_I$, is an extended metric on the set of isomorphism classes of finitely presentable $n$-parameter persistence modules.
For details and more background, see \cref{section:background}.
We discuss some implications of our main results in the next subsection.


\begin{restatable}{theoremx}{indecomposabledense}
    \label{theorem:indecomposables-dense}
    Let $n \geq 2$.
    In the space of finitely presentable $n$-parameter persistence modules equipped with the interleaving distance, the indecomposable modules form a dense subset.
\end{restatable}

Our second main result relates the property of being close to an indecomposable to the property of decomposing as a direct sum of an indecomposable and a nearly trivial persistence module.
In order to state it, we need the following two definitions.
Let $\epsilon > 0$.
A multi-parameter persistence module $M : \Rbf^n \to \vect$ is \define{strictly $\epsilon$-trivial} if there exists $\epsilon' < \epsilon$ such that the structure morphism $M(r) \to M(r+\epsilon')$ is zero for every $r \in \Rbf^n$.
A module is \define{$\epsilon$-indecomposable} if it decomposes as a direct sum of an indecomposable and a strictly $\epsilon$-trivial module.

\begin{restatable}{theoremx}{stabilityindecomposability}
    \label{proposition:stability-indecomposability}
    Let $n \geq 1$ and $\epsilon > 0$.
    In the space of finitely presentable $n$-parameter persistence modules equipped with the interleaving distance, 
    the $\epsilon$-indecomposable modules form an open set.
\end{restatable}


\begin{restatable}{cor}{maintheorem}
    \label{theorem:main-theorem}
    Let $n \geq 2$ and $\epsilon > 0$.
    In the space of finitely presentable $n$-parameter persistence modules equipped with the interleaving distance, 
    the $\epsilon$-indecomposable modules form a dense and open set.
\end{restatable}

The second corollary concerns the notion of a structurally stable persistence module, which we now define.
Informally, a module is structurally stable if its decomposition into indecomposables is stable under perturbations in the interleaving distance.

We first define the notion of matching between (the full decompositions of) two persistence modules.
Let $\epsilon \geq 0 \in \Rbb$, and let $M, N : \Rbf^n \to \vect$.
An \emph{$\epsilon$-matching} between $M$ and $N$ consists of two decompositions $M \cong \bigoplus_{j \in J} M_j$ and $N \cong \bigoplus_{j \in J} N_j$, indexed by the same set~$J$, where each summand is either indecomposable or zero,
and such that $M_j$ and $N_j$ are $\epsilon$-interleaved for every $j \in J$.
We then define the \emph{bottleneck distance} between $M$ and $N$ as
\[
    d_B(M,N) = \inf \{\epsilon \geq 0 : \text{there exists an $\epsilon$-matching between $M$ and $N$}\},
\]
and say that a finitely presentable module $M : \Rbf^n \to \vect$ is \emph{structurally stable} if, for every $\epsilon > 0$, there exists $\delta > 0$ such that every finitely presentable $N : \Rbf^n \to \vect$ with $d_I(M,N) < \delta$ satisfies $d_B(M,N) < \epsilon$.

In other words,
$M$ is structurally stable if any neighborhood of it in the bottleneck topology is also a neighborhood in the interleaving topology.
The converse (that neighborhoods in the interleaving topology are also neighborhoods in the bottleneck topology) is always true: one has $d_I(M,N) \leq d_B(M,N)$ since $\epsilon$-interleavings between the summands of decompositions of $M$ and $N$ can be summed together to get an $\epsilon$-interleaving between $M$ and~$N$.

As a motivating example, we recall that the isometry theorem 
\cite{
    lesnick,
    bubenik-scott,
    bauer-lesnick,
    chazal-silva-glisse-oudot,
}
states that, for one-parameter persistence modules,
two modules are $\epsilon$-interleaved if and only if their decompositions into indecomposables are $\epsilon$-matched, implying that $d_I = d_B$.
This means that every finitely presentable one-parameter persistence module is structurally stable in a very strong sense.
In particular, for $\epsilon = 0$ this theorem recovers the Krull--Schmidt property of one-parameter persistence modules, asserting the essential uniqueness of the decomposition into indecomposables.

If every finitely presentable multi-parameter persistence module were structurally stable, then $d_I$ and~$d_B$ would induce the same topology.
This is far from true.

\begin{restatable}{cor}{structuralstability}
    \label{corollary:structural-stability}
    Let $n \geq 2$.
    A finitely presentable $n$-parameter persistence module is structurally stable if and only if it is indecomposable or zero.
\end{restatable}


The last two corollaries concern the metric approximation of arbitrary multi-parameter persistence modules by
direct sums of modules in 
some specified family of indecomposables~$\Fcal$.
The first one says that, if all modules in an open set can be approximated by $\Fcal$-decomposable modules, then those same modules can already be approximated by modules in $\Fcal$.

\begin{restatable}{cor}{denseFdecomposables}
    \label{theorem:dense-F-decomposables}
Let $n\geq 2$.
Let $\Fcal$ be a collection of isomorphism classes of indecomposable, finitely presentable $n$-parameter persistence modules.
The interior of the closure of the set of $\Fcal$-decomposable modules is equal to the interior of the closure of $\Fcal$.
In particular, if the $\Fcal$-decomposable modules are dense, then the set $\Fcal$ must itself be dense.
\end{restatable}

The final corollary gives a precise sense in which any family of indecomposables $\Fcal$ with the property that the $\Fcal$-decomposable are dense must be very complicated.
In particular, it implies that multi-parameter spread decomposable modules \cite{blanchette-brustle-hanson} (sometimes called interval decomposable modules \cite{botnan-lesnick}) are nowhere dense, in contrast to the one-parameter case, where every module is interval decomposable.

\begin{restatable}{cor}{denseFcomplicated}
    \label{theorem:dense-F-complicated}
Let $n\geq 2$.
Let $\Fcal$ be a collection of isomorphism classes of indecomposable, finitely presentable $n$-parameter persistence modules.
If $\Fcal$ does not contain indecomposables of arbitrarily high pointwise dimension, then the set of $\Fcal$-decomposable modules is nowhere dense.
\end{restatable}

\subparagraph{Discussion}


To put \cref{theorem:main-theorem} into context, we mention a similar instance of a set of persistence modules being open and dense:
The set of all finitely presentable one-parameter persistence modules that are \emph{staggered} (meaning that the finite endpoints of the intervals in their decomposition into indecomposables are all distinct) is dense but not open;
nevertheless, for every $\epsilon > 0$ the set of modules that decompose as $A \oplus B$ with $A$ staggered and $B$ strictly $\epsilon$-trivial is dense and open, as a direct consequence of the isometry theorem.

Since there are no general relationships between meager and Borel measure-zero sets~\cite{oxtoby}, \cref{theorem:main-theorem} does not imply that random $n$-parameter persistence modules, according to some suitable probability measure, are $\epsilon$-indecomposable.
One could then ask the following questions:
What are interesting and useful probability measures on spaces of multi-parameter persistence modules?
According to these probability measures, how do random multi-parameter persistence modules decompose?
Partial answers to these questions appear in \cite{alonso-kerber}, but, to the best of our knowledge, these questions are otherwise largely open.


\medskip

\begin{rem}[On the meaning of genericity]
Recall that, in a topological space, a generic property is one that holds for all points of some dense and open set.
Thus, \cref{theorem:main-theorem} says that, in the space of finitely presentable $n$-parameter persistence modules ($n\geq 2$) endowed with the interleaving distance, being $\epsilon$-indecomposable is generic, for every fixed $\epsilon > 0$.

\cref{theorem:main-theorem} also implies that, in the same space, the indecomposable modules form a \emph{residual set}, that is, a subset that can be written as a countable intersection of open dense sets.
The term \emph{generic} is also used for residual sets, but this is typically done only in the context of \emph{Baire spaces}, defined by the property that every residual set is dense.
The purpose of this restriction in terminology is to exclude the possibility that a set and its complement are both generic. In particular, if a space is meagre (a countable union of nowhere dense subsets), then each of its subsets is both residual and meagre, but not necessarily dense.
In fact, it is possible that a subset of a meagre space and its complement are both dense. For example, the subset of the rational numbers having an even numerator in its reduced fraction is dense and has a dense complement.

Since the space of finitely presented $n$-parameter persistence modules is not Baire (\cref{lemma:not-baire}), in the first version of this paper we avoided using the term generic when refering to indecomposable modules, and we raised the question whether there exists a well-behaved Baire space of $n$-parameter persistence modules in which the finitely presentable modules form a residual (and hence dense) subset.
As it turns out, such a Baire space can be constructed using results to appear in \cite{bauer-gusel-scoccola}; the upshot being that indecomposability is a generic property in a suitable Baire space which contains the finitely presented modules.
We now give more details.


\smallskip

In the rest of this remark, we only consider $n$-parameter persistence modules with $n\geq 2$.
We start with a few definitions.
Let $M$ be an $n$-parameter persistence module.
We say that $M$ is \emph{q-tame} (respectively \emph{ephemeral}) if for every $r,s \in \Rbf^n$ such that $r_i < s_i$ for all $1 \leq i \leq n$, we have that the structure morphism $M(r) \to M(s)$ has finite rank (respectively is zero).
We say that $M$ is \emph{lower-semicontinuous} (respectively \emph{upper-semicontinuous}) if the canonical map $\colim_{r'<r}M(r') \to M(r)$ (respectively $M(r) \to \lim_{r < r'} M(r')$) is an isomorphism.

Following \cite{chazal-crawleybovey-desilva,berkouk-petit}, let $\Ocal$ be the \emph{observable category} of the category of q-tame $n$-parameter persistence modules, that is, the additive category obtained by taking the quotient of the category of q-tame $n$-parameter persistence modules by the subcategory spanned by the ephemeral $n$-parameter persistence modules.

\smallskip
\noindent\textit{Result 1} (in \cite{bauer-gusel-scoccola}): 
The category $\Ocal$ is equivalent to
the category of q-tame and lower semi-continuous $n$-parameter persistence modules,
and to
the category of q-tame and upper semi-continuous $n$-parameter persistence modules.
In both cases, the equivalence is given by embedding q-tame, semi-continuous modules into q-tame modules, and then taking the quotient by ephemerals.

\smallskip

\noindent\textit{Result 2} (in \cite{bauer-gusel-scoccola}): 
The category $\Ocal$ is Krull--Schmidt, meaning that every object of $\Ocal$ has an essentially unique decomposition as a direct sum of indecomposable objects.
\smallskip

Let $\Ocal_{met}$ be the set of isomorphism classes of objects of $\Ocal$ (which can be seen to be a set, as opposed to a proper class), endowed with the interleaving distance.

\smallskip
\noindent\textit{Result 3} (in \cite{bauer-gusel-scoccola}): 
The extended pseudometric space $\Ocal_{met}$ is in fact an extended metric space, and it is metrically complete.
\smallskip

Hence, by Results 2 and 3, as a topological space $\Ocal_{met}$ is Baire, and every element of it has a well-defined and unique decomposition into indecomposables.
Note that this relies on the fact that $\Ocal_{met}$ is an extended metric space and not just an extended pseudometric space, so that two objects $\Ocal$ are at interleaving distance zero if and only if they are isomorphic.
Now, let $\Fcal_{met} \subseteq \Ocal_{met}$ be the metric closure of the set of (isomorphism classes of) finitely presentable modules in $\Ocal_{met}$.

\smallskip
\noindent\textit{Result 4} (same proof as \cref{proposition:stability-indecomposability}):
For every $\epsilon > 0$ and for every indecomposable, finitely presented persistence module $M$, there exists an open neighborhood of $M$ in $\Fcal_{met}$ only containing $\epsilon$-trivial modules (that is modules decomposing as a direct sum of an indecomposable and a module at interleaving distance strictly less than $\epsilon/2$ from $0$).
\smallskip

The space $\Fcal_{met}$ thus has the following properties: it is Baire; every element of it has a well-defined, unique decomposition into indecomposables; for every $\epsilon > 0$, the $\epsilon$-indecomposables are open (Result 4) and dense (\cref{theorem:indecomposables-dense}).
By taking $\epsilon = 1/n$, we see that (the isomorphism classes of) finitely presented indecomposables form a residual set of $\Fcal_{met}$, so that, in $\Fcal_{met}$, being indecomposable is a generic property.
Finally, note that $\Fcal_{met}$ can be realized as an actual space of persistence modules: by Result 1, one can take the space of isomorphism classes of q-tame and lower semi-continuous (or alternatively upper-semicontinuous) persistence modules that can be approximated arbitrary well (in the interleaving distance) by finitely presented modules.
\end{rem}

\medskip

Although \cref{corollary:structural-stability} says that the indecomposable modules, together with the zero module, are the only multi-parameter persistence modules for which decomposition into indecomposables is a stable operation, recent work of Bjerkevik \cite[Theorem~5.3~and~Conjecture~6.1]{bjerkevik2023stabilizing} shows that decomposition into indecomposables may behave much better if one is allowed to consider certain subquotients of the modules.
Note that computing the indecomposable summands of multi-parameter persistence modules is feasible in practice \cite{qpa,dey-xin}.

\medskip

For finite posets, having a global bound on the pointwise dimension of all indecomposable representations (\emph{pointwise finite representation type}) is equivalent to admitting only a finite number of isomorphism types of indecomposable representations (\emph{finite representation type}).
This follows from Roiter's theorem \cite{roiter};
note also that, for a finite poset, the property of having finite representation type is independent of the field, by~\cite{loupias}.
Although there are infinitely many isomorphism classes of indecomposable one-parameter persistence modules, the infinite poset $\Rbf$ is still of pointwise finite representation type, since the only indecomposables are the interval modules, which have pointwise dimension at most one \cite{crawley-boevey}.
Since $\Rbf^n$ ($n\geq 2$) contains finite full subposets of infinite representation type, it is clear that it is not of pointwise finite representation type.
\cref{theorem:dense-F-complicated} strengthens this last fact by stating that
any subcategory
of finitely presentable $n$-parameter persistence modules that is dense in the interleaving distance cannot be of pointwise finite representation type.


\medskip

In this paper, we focus on the topology induced by the interleaving distance, since this is the most widely used choice of distance between multi-parameter persistence modules.
Nevertheless, several other distances have been proposed \cite{scolamiero-chacholski-lundman-ramanujan-oberg, bjerkevik-lesnick,giunti-nolan-otter-waas,bubenik-scott-stanley}, and it would be interesting to understand the topological properties of the set of indecomposables when other distances are used.
In this direction, we believe that the main results in this paper also hold when the interleaving distance is replaced with any of the presentation distances of Bjerkevik and Lesnick \cite{bjerkevik-lesnick}.

\subparagraph{Structure of the paper}
In \cref{section:background}, we recall necessary background and introduce notation.
In \cref{section:indecomposables-dense}, we prove \cref{theorem:indecomposables-dense}.
In \cref{section:almost-indecomposable-open}, we prove \cref{proposition:stability-indecomposability}.
In \cref{section:consequences}, we prove \cref{theorem:main-theorem,corollary:structural-stability,theorem:dense-F-decomposables,theorem:dense-F-complicated}.
In \cref{appendix}, we give proofs of known results.

%

\subparagraph{Acknowledgements}
We thank H\r{a}vard Bjerkevik, Barbara Giunti, and Francesca Tombari for interesting conversations and useful observations.
L.S.~thanks
Nicolas Berkouk,
Mathieu Carrière,
René Corbet,
Christian Hirsch,
Claudia Landi,
Vadim Lebovici,
David Loiseaux,
Ezra Miller,
Steve Oudot,
François Petit,
and Alexander Rolle
for discussions during the \emph{Metrics in Multiparameter Persistence workshop} (Lorentz Center, 2021).
This research has been conducted while U.B.~was participating in the program \emph{Representation Theory: Combinatorial Aspects and Applications};
we thank the Centre for Advanced Study (CAS) at the Norwegian Academy of Science and Letters for their hospitality and support.
U.B.~was supported by the German Research Foundation (DFG) through the Collaborative Research Center SFB/TRR 109 \emph{Discretization in Geometry and Dynamics} -- 195170736.
L.S.~was partially supported by the National Science Foundation through grant CCF-2006661 and CAREER award DMS-1943758, as well as by EPSRC grant ``New Approaches to Data Science: Application Driven Topological Data Analysis'', EP/R018472/1.
For the purpose of Open Access, the authors have applied a CC BY public copyright licence to any Author Accepted Manuscript (AAM) version arising from this submission.

\section{Background}
\label{section:background}

Although we use some notions from category theory, we only assume familiarity with basic concepts; in particular: categories, isomorphisms, functors, functor categories, direct sums, and kernels and cokernels.
See, e.g., \cite{maclane} for an introduction.

Throughout the paper, we fix a field $\kbb$ and let $\vect$ denote the category of finite dimensional $\kbb$-vector spaces.

An \define{extended metric space} consists of a set $A$ together with a function $d : A \times A \to \Rbb \cup \{\infty\}$ such that, for all $a,b \in A$ we have $d(a,b) = d(b,a) \geq 0$ and $d(a,b) = 0$ if and only if $a=b$; and for all $a,b,c \in A$ we have $d(a,c) \leq d(a,b) + d(b,c)$, with the convention that $r + \infty = \infty + r = \infty$ for all $r \in \Rbb \cup \{\infty\}$.

\subparagraph{Posets}
We let $\Rbf$ denote the poset of real numbers with its standard ordering and reserve the notation $\Rbb$ for the metric space of real numbers.
Let $n \in \Nbb$ and consider the product poset $\Rbf^n$.
By an abuse of notation, if $r \in \Rbf$, we interpret it as an element $r \in \Rbf^n$ all of whose coordinates are equal to $r$.
Thus, if for instance $s \in \Rbf^n$ and $r \in \Rbf$, then $s + r \in \Rbf^n$ denotes the element $(s_1 + r, s_2 + r, \dots, s_n + r)$. 
We let $\Zbf \subseteq \Rbf$ denote the full subposet spanned by the integers.
If $n \in \Nbb$ and $1 \leq k \leq n$, we denote the $k$th \emph{standard basis vector} by $\ebf_k \in \Zbf^n \subseteq \Rbf^n$.

We interpret any poset $\Pscr$ as a category with objects the elements of the poset and a unique morphism $i \to j \in \Pscr$ whenever $i \leq j$ and no morphism otherwise.

\subparagraph{Persistence modules}
Let $\Pscr$ be a poset.
A pointwise finite dimensional \define{$\Pscr$-persistence module} is a functor $M : \Pscr \to \vect$.
Note that all persistence modules in this paper are assumed to be pointwise finite dimensional and that, for the sake of brevity, we omit this qualifier.
When the indexing poset $\Pscr$ is clear from the context, we may refer to a $\Pscr$-persistence module simply as a persistence module or as a module.
The collection of all $\Pscr$-persistence modules forms a category, where the morphisms are given by natural transformations.

If $M : \Pscr \to \vect$ is a $\Pscr$-persistence module and $i \leq j \in \Pscr$, we let $\phi^M_{i,j} : M(i) \to M(j)$ denote the structure morphism corresponding to the morphism $i \to j$ in $\Pscr$ seen as a category.

Let $\Pscr$ be a poset and let $i \in \Pscr$.
Define the persistence module $\Psf_i : \Pscr \to \vect$ by
\[
    \Psf_i(j) =
    \begin{cases}
        \kbb & i \leq j\\
        0 & \text{else},
    \end{cases}
\]
with all structure morphisms that are not forced to be zero being the identity $\kbb \to \kbb$.

A $\Pscr$-persistence module is \define{finitely presentable} if it is isomorphic to the cokernel of a morphism $\bigoplus_{j \in J} \Psf_j \to \bigoplus_{i \in I} \Psf_i$, where $I$ and $J$ are finite multisets of elements of $\Pscr$.
This is consistent with the standard nomenclature in representation theory and algebra since a $\Pscr$-persistence module is finitely generated and projective precisely if it is isomorphic to a finite direct sum of modules of the form $\Psf_i$.

\subparagraph{Restrictions and extensions}
If $\Qscr \subseteq \Pscr$ is an inclusion of posets and $M : \Pscr \to \vect$ is a $\Pscr$-persistence module, the \define{restriction} of $M$ to $\Qscr$, denoted $M|_\Qscr : \Qscr \to \vect$, is the $\Qscr$-persistence module obtained by precomposing $M : \Pscr \to \vect$ with the inclusion $\Qscr \to \Pscr$.

We consider two main types of subposets of $\Rbf^n$.
In one case, we let $\{r_1 < r_2 < \dots < r_k\}$ be a finite set of real numbers and consider the product poset $\{r_1 < r_2 < \dots < r_k\}^n \subseteq \Rbf^n$.
We refer to a subposet of $\Rbf^n$ obtained in this way as a \define{finite grid}.
In the other case, we let $\{r_i\}_{i \in \Zbb}$ be a countable set of real numbers without accumulation points and such that $r_i < r_{i+1}$ for all $i \in \Zbb$, and again consider the product poset $\{r_i\}^n \subseteq \Rbf^n$.
We refer to a subposet of $\Rbf^n$ obtained in this way as a \define{countable grid}.
A \define{regular grid} is any countable grid of the form $(s \Zbf)^n = \{s \cdot m : m \in \Zbf\}^n \subseteq \Rbf^n$ for $s > 0 \in \Rbb$, where $\Zbf$ denotes the poset of integers.
Note that, in the definitions of finite and countable grid, one could take a product of different subposets of $\Rbf$, instead of an $n$-fold product of the same poset.
Since we do not need this generality, we choose the more restrictive definition for simplicity.

Let $\Pscr \subseteq \Rbf^n$ be a finite or countable grid.
Given $M : \Pscr \to \vect$ define $\widehat{M} : \Rbf^n \to \vect$
by
\[
    \widehat{M}(r) =
    \begin{cases}
        M\left(\sup\{ p \in \Pscr : p \leq r\}\right) & \text{if %
        $\{ p \in \Pscr : p \leq r\} \neq \emptyset$},\\
        0 & \text{else};
    \end{cases}
\]
for its structure morphisms use the structure morphisms of $M$ and the fact that $\sup\{p \in \Pscr : p \leq r\} \leq \sup\{p \in \Pscr : p \leq s\}$ whenever $r \leq s$ and there exists $p \in \Pscr$ such that $p \leq r$.
We refer to $\widehat{M}$ as the \define{extension} of $M$ along the inclusion $\Pscr \subseteq \Rbf^n$.
As a side remark, we note that this notion of extension is an instance of the notion of left Kan extension from category theory, but we do not make use of this fact.

A module $M : \Rbf^n \to \vect$ is a \emph{$\Pscr$-extension} if there exists a module $N : \Pscr \to \vect$ such that $M \cong \widehat{N}$.
%
If $\Pscr \subseteq \Rbf^n$ is a finite or countable grid and $M : \Rbf^n \to \vect$, define the \define{restriction-extension} of $M$ along $\Pscr$, denoted $M_\Pscr : \Rbf^n \to \vect$, as the extension along $\Pscr \subseteq \Rbf^n$ of the restriction $M|_\Pscr : \Pscr \to \vect$.
Given $M,N : \Rbf^n \to \vect$ and a morphism $f : M \to N$, there is a morphism $f_\Pscr : M_\Pscr \to N_\Pscr$ given by extending the restriction $f|_\Pscr : M|_\Pscr \to N|_\Pscr$.
It follows immediately that this construction is functorial in the sense that given modules $A,B,C : \Rbf^n \to \vect$ and morphism $f : A \to B$ and $g : B \to C$, we have $g_\Pscr \circ f_\Pscr = (g \circ f)_\Pscr : A_\Pscr \to C_{\Pscr}$.


\begin{lem}
    \label{lemma:fp-is-extension-finite-poset}
    Let $n \geq 1$ and let $M : \Rbf^n \to \vect$.
    Then $M$ is finitely presentable if and only if there exists a finite grid $\Pscr \subseteq \Rbf^n$ such that $M \cong M_\Pscr$.
\end{lem}
\begin{proof}
    Assume that $M$ is finitely presentable, so that $M$ is isomorphic to the cokernel of a morphism $\bigoplus_{j \in J} \Psf_j \to \bigoplus_{i \in I} \Psf_i$ with $I$ and $J$ finite subsets of $\Rbf^n$.
    Consider the set $S = \{x_k \in \Rbf : x \in I \cup J, \, 1 \leq k \leq n\}$ of all the coordinates of the points in $I \cup J$, and the finite grid $\Pscr = S^n \subseteq \Rbf^n$.
    It is straightforward to see that $M \cong M_\Pscr$.

    Now assume that $M \cong \widehat{N}$ with $\Pscr \subseteq \Rbf^n$ a finite grid and $N : \Pscr \to \vect$.
    Since the poset $\Pscr$ has finitely many elements and $N$ is pointwise finite dimensional, there exists an epimorphism $e : \bigoplus_{i \in I} \Psf_i \to N$, for some finite multiset $I$ of elements of $\Pscr$.
    Similarly, if $K$ is the kernel of the morphism $e$, there must exist an epimorphism $\bigoplus_{j \in J} \Psf_j \to K$, with $J$ finite.
    Putting these two morphisms together we see that $N$ is isomorphic to the cokernel of a morphism $\bigoplus_{j \in J} \Psf_j \to \bigoplus_{i \in I} \Psf_i$.
    It is easy to see that $M$ is then isomorphic to the cokernel of the induced morphism $\bigoplus_{j \in J} \widehat{\Psf_j} \to \bigoplus_{i \in I} \widehat{\Psf_i}$, and that $\widehat{\Psf_r} = \Psf_r : \Rbf^n \to \vect$ for all $r \in \Pscr$.
\end{proof}

It is straightforward to see that, in a grid extension persistence module, the structure maps that do not cross the grid are isomorphisms, as recorded in the following lemma.

\begin{lemma}
    \label{lemma:structure-map-iso}
    Let $\Pscr$ be a finite or countable grid and let $M : \Rbf^n \to \vect$ be a $\Pscr$-extension.
    Let $r < s \in \Rbf^n$ such that every $p \in \Pscr$ with $p \leq s$ also satisfies $p \leq r$.
    Then the structure morphism $\phi^M_{r,s} : M(r) \to M(s)$ is an isomorphism.
\end{lemma}

\subparagraph{Decomposition of persistence modules}
The proofs of the results in this section are standard; we include them in \cref{appendix}, for completeness.
Let $\Pscr$ be a poset and let $M : \Pscr \to \vect$.
We say that $M$ is \define{decomposable} if there exist $A,B : \Pscr \to \vect$ non-zero such that $M \cong A \oplus B$.
If $M$ is non-zero and not decomposable, we say that $M$ is \define{indecomposable}.


The next two results follow from \cite{botnan-crawleybovey} and the Krull--Remak--Schmidt--Azumaya theorem~\cite{azumaya}.


\begin{restatable}{thm}{theoremdecomposition}
    \label{theorem:decomposition}
    Let $\Pscr$ be a poset and let $M : \Pscr \to \vect$ be a pointwise finite dimensional $\Pscr$-persistence module.
    There exists a set $I$ and an indexed family of indecomposable $\Pscr$-persistence modules $\{M_i\}_{i \in I}$ such that $M \cong \bigoplus_{i \in I} M_i$.
    Moreover, if $M \cong \bigoplus_{j \in J} M_j$ for another indexed family of indecomposable $\Pscr$-persistence modules $\{M_j\}_{j \in J}$, then there exists a bijection $f : I \to J$ such that $M_i \cong M_{f(i)}$ for all $i \in I$.
\end{restatable}

\begin{restatable}{lemma}{indecomposablelocalring}
    \label{lemma:indecomposable-local-ring}
    Let $\Pscr$ be a poset.
    A persistence module $M : \Pscr \to \vect$ is indecomposable if and only if its endomorphism ring $\End(M)$ is local.
\end{restatable}

\begin{proof}
    This is a direct consequence of \cite[Theorem~1.1]{botnan-crawleybovey}.
\end{proof}

The following result states that direct sum decompositions behave well with respect to restrictions and extensions; its proof is straightforward.

\begin{lem}
    \label{lemma:decomposition-restriction-extension}
    Let $\Pscr \subseteq \Rbf^n$ be a subposet.
    Let $M : \Rbf^n \to \vect$ decompose as $M \cong \bigoplus_{i\in I} M_i$ for some indexed family $\{M_i\}_{i \in I}$ of persistence modules.
    Then $M|_\Pscr \cong \bigoplus_{i\in I} (M_i)|_\Pscr$.
    Similarly, if $\Pscr$ is a finite or countable grid and $M : \Pscr \to \vect$ decomposes as $M \cong \bigoplus_{i\in I} M_i$ for some indexed family $\{M_i\}_{i \in I}$ of $\Pscr$-persistence modules, then $\widehat{M} \cong \bigoplus_{i\in I} \widehat{M_i}$.
\end{lem}

The next result asserts that finitely presentable persistence modules decompose as a finite direct sum of indecomposable, finitely presentable modules.

\begin{restatable}{lem}{decompositionfp}
    \label{lemma:decomposition-fp}
    Let $n \geq 1$ and let $M : \Rbf^n \to \vect$.
    If $M$ is finitely presentable, then $M \cong \bigoplus_{i \in I} M_i$ with $\{M_i\}_{i \in I}$ a finite family of finitely presentable indecomposable $\Rbf^n$-persistence modules.
\end{restatable}

\subparagraph{Interleavings, interleaving distance, and space of persistence modules}
Let $n \in \Nbb$.
If $M : \Rbf^n \to \vect$ is an $\Rbf^n$-persistence module and $r \in \Rbf$, the \define{$r$-shift} of $M$ is the persistence module $M[r] : \Rbf^n \to \vect$ with $M[r](s) \coloneqq M(s+r)$ for all $s \in \Rbf^n$ and $\phi^{M[r]}_{s,t} \coloneqq \phi^M_{s+r,t+r}$ for all $s\leq t \in \Rbf^n$.
If $r \geq 0 \in \Rbf$, then there is a natural transformation $\eta^M_r : M \to M[r]$ with component $M(a) \to M[r](a) = M(a+r)$ given by the structure morphism $\phi^M_{a,a+r}$.

Let $M,N : \Rbf^n \to \vect$ and $\epsilon \geq 0 \in \Rbb$.
An \define{$\epsilon$-interleaving} between $M$ and $N$ consists of a pair of morphisms $f : M \to N[\epsilon]$ and $g : N \to M[\epsilon]$ such that $g[\epsilon] \circ f = \eta^M_{2\epsilon}$ and $f[\epsilon] \circ f = \eta^N_{2\epsilon}$.
The \define{interleaving distance} between $M$ and $N$ is
\[
    d_I(M,N) = \inf \{\epsilon \geq 0 : \text{there exists an $\epsilon$-interleaving between $M$ and $N$}\}.
\]

By composing interleavings, one shows that $d_I$ satisfies the triangle inequality.
Using \cref{lemma:structure-map-iso}, one sees that if $M, N : \Rbf^n \to \vect$ are finitely presentable and $d_I(M,N) = 0$, then $M$ and $N$ are isomorphic \cite[Corollary~6.2]{lesnick}.
This implies that $d_I$ defines an extended metric on the set of isomorphism classes of finitely presentable $\Rbf^n$-persistence modules.
We mention here that the collection of all isomorphism classes of finitely presentable $\Rbf^n$-persistence modules is indeed a set, as opposed to a proper class, but we shall not delve into the details as they are standard and not relevant to the results presented here.

The \define{space of finitely presentable $n$-parameter persistence modules} is the topological space with underlying set the isomorphism classes of finitely presentable $\Rbf^n$-persistence modules and with topology generated by the open balls with respect to the interleaving distance.

Let $\epsilon > 0$.
A persistence module $M : \Rbf^n \to \vect$ is \define{$\epsilon$-trivial} if, for every $r \in \Rbf^n$, the structure morphism $\phi^M_{r,r+\epsilon} : M(r) \to M(r + \epsilon)$ is the zero morphism.
Note that this is equivalent to $M$ being $\frac\epsilon2$-interleaved with the zero module.
Thus, if $d_I(M,0) < \frac\epsilon2$, then $M$ is $\epsilon$-trivial.
More stringently, the module $M$ is \define{strictly $\epsilon$-trivial} if it is $\epsilon'$-trivial for some $\epsilon' < \epsilon$.

\subparagraph{Shifts and restriction-extensions}
If $\Pscr = (\{s_i\}_{i \in \Zbb})^n \subseteq \Rbf^n$ is a countable grid and $r \in \Rbf$, define the \define{shifted} countable grid $\Pscr + r = (\{s_i + r\}_{i \in \Zbb})^n \subseteq \Rbf^n$.
Let $M : \Rbf^n \to \vect$ and $s \in \Rbf$.
If $a \in \Rbf^n$, then $\sup\{ p + r : p + r \leq a + s\} = \sup\{p + r - s : p + r - s \leq a\} + s$.
Thus,
\[
    M_{\Pscr + r}[s] = M[s]_{\Pscr + (r-s)},
\]
a fact that we use repeatedly in \cref{section:almost-indecomposable-open}.
In particular, given $N : \Rbf^n \to \vect$ and a morphism $f : M \to N[s]$, we get a morphism
\begin{equation}
    \label{equation:shift-and-restriction-extension}
    f_\Pscr : M_{\Pscr + r} \to N[s]_{\Pscr + r} = N_{\Pscr + r + s}[s].
\end{equation}

\section{Indecomposables are dense}
\label{section:indecomposables-dense}

The proof of \cref{theorem:indecomposables-dense} depends on the Tacking Lemma (\cref{lemma:main-lemma-attaching}), which lets us ``tack together'' two indecomposable modules to obtain a third indecomposable module that is at small interleaving distance from the direct sum of the initial modules.
This section is structured as follows.
In \cref{section:proof-thm-A}, we state the Tacking Lemma without proof and use it to prove \cref{theorem:indecomposables-dense}.
In \cref{section:gluing-functors}, we develop the main techniques needed to prove the Tacking Lemma.
In \cref{section:tacking-lemma}, we prove the Tacking Lemma.

\subsection{Proof of \cref{theorem:indecomposables-dense}}
\label{section:proof-thm-A}

\begin{restatable}[Tacking Lemma]{lem}{mainlemma}
    \label{lemma:main-lemma-attaching}
    Let $n \geq 2$.
    Let $A,B : \Rbf^n \to \vect$ be indecomposable, finitely presentable persistence modules, and assume that $A$ and $B$ are $\Pscr$-extensions for some regular grid $\Pscr \subseteq \Rbf^n$.
    For every $\delta > 0$ there exists a regular grid $\Qscr \supseteq \Pscr$ and a persistence module $M : \Rbf^n \to \vect$ with $M$ indecomposable, finitely presentable, a $\Qscr$-extension, and such that $d_I(M, A \oplus B) \leq \delta$.
\end{restatable}

We also need the following lemma, whose proof is immediate from the definition of interleaving distance.

\begin{lem}
         \label{lemma:interleaving-stabilization}
         Let $A,B,M : \Rbf^n \to \vect$. Then
$d_I(A \oplus M,B \oplus M) \leq d_I(A,B)$.
\end{lem}

\indecomposabledense*
\begin{proof}
    Let $N : \Rbf^n \to \vect$ be finitely presentable and let $\epsilon > 0$.
    We show that there exists an indecomposable, finitely presentable persistence module $M : \Rbf^n \to \vect$ such that $d_I(M,N) \leq \epsilon$.
    We start by reducing the problem to the case in which the given module is an extension of a module defined over a regular grid.

    Consider the regular grid $\Pscr = (\frac\epsilon2\Zbf)^n$ and note that $d_I(N_\Pscr,N) \leq \frac\epsilon2$.
    Let $L = N_\Pscr$.
    It is easy to see that $L$ is finitely presentable.

    If $L = 0$, then it seems natrual to take $M$ to also be $0$.
    However, since the zero module is not indecomposable, we take $M$ to be
    any strictly $\epsilon$-trivial indecomposable and finitely presentable persistence module; for example,
    the persistence module such that $M(r) = \kbb$ if $0 \leq r < \epsilon$, and $0$ otherwise, with all the structure morphisms that are not forced to be the zero morphism being the identity $\kbb \to \kbb$.
    We have $d_I(M,0) = \frac\epsilon2$, and the claim now follows from $d_I(N,M) \leq d_I(N,L) + d_I(L,M) \leq \frac\epsilon2 + d_I(0,M) \leq \epsilon$.

    We now assume that $L$ is non-zero.
    By \cref{lemma:decomposition-fp}, there exist finitely presentable indecomposables $X_1, \dots, X_k : \Rbf^n \to \vect$ such that $L \cong X_1 \oplus \dots \oplus X_k$.
    Note that, for all $1 \leq i \leq k$, the persistence module $X_i$ is a $\Qscr$-extension for any grid $\Qscr \supseteq \Pscr$.
    Let us define $M_i : \Rbf^n \to \vect$ for $1 \leq i \leq k$ inductively.
    Let $M_1 = X_1$ and, for $2 \leq i \leq k$, let $M_{i}$ be obtained by setting $A = M_{i-1}$, $B = X_i$, and $\delta = \frac{\epsilon}{2k}$ in \cref{lemma:main-lemma-attaching}.
    Thus, $M_i$ is indecomposable, finitely presentable, and satisfies
    \begin{equation}
        \label{equation:attached-distance}
        d_I(M_i, M_{i-1} \oplus X_i) \leq \frac{\epsilon}{2k}.
    \end{equation}
    Let $M = M_k$.
    We now prove that $d_I(M_i, X_1 \oplus \cdots \oplus X_i) \leq \frac{i\epsilon}{2k}$ for all $1 \leq i \leq k$, by induction on $i$.
    If $i = 1$, this is clear.
    Otherwise, we have
    \begin{align*}
        d_I(M_i, X_1 \oplus \dots \oplus X_i) &\leq
        d_I(M_i, M_{i-1} \oplus X_i) + d_I(M_{i-1} \oplus X_i, X_1 \oplus \dots \oplus X_i)\\
        &\leq
        d_I(M_i, M_{i-1} \oplus X_i) + d_I(M_{i-1}, X_1 \oplus \dots \oplus X_{i-1})\\
        &\leq
        \frac{\epsilon}{2k} + d_I(M_{i-1}, X_1 \oplus \dots \oplus X_{i-1})\\
        &\leq
        \frac{\epsilon}{2k} + (i-1)\frac{\epsilon}{2k} = \frac{i\epsilon}{2k},
    \end{align*}
    where in the first inequality we used the triangle inequality, in the second inequality we used \cref{lemma:interleaving-stabilization} for the second summand, in the third inequality we used \cref{equation:attached-distance}, and in the last inequality we used the inductive hypothesis.

    In particular $d_I(M,L) = d_I(M_k,X_1 \oplus \cdots \oplus X_k) \leq \frac{k\epsilon}{2k} = \frac{\epsilon}{2}$, so, by the triangle inequality, we get $d_I(M,N) \leq d_I(M,L) + d_I(L,N) \leq \epsilon$, as required.
\end{proof}

\subsection{Gluing functors on subposets}
\label{section:gluing-functors}

Recall that a poset is \emph{connected} if it is connected as a category.
Explicitly, $\Pscr$ is connected if it is non-empty, and for every pair of elements $x,y \in \Pscr$ there exists a sequence of elements $p_1, \dots, p_k \in \Pscr$ with $x = p_1$, $y = p_k$ and such that for every $1 \leq i \leq k-1$ either $p_i \leq p_{i+1}$ or $p_i \geq p_{i+1}$ holds.
Although not stated in this language, the following lemma gives sufficient conditions for a poset to be a pushout (in the category of small categories) of a pair of its subsets.
In the lemma, all subsets of a poset are understood as full subposets.

\begin{lem}
    \label{lemma:pushout-subposets}
    Let $\Ccal$ be a category, let $\Pscr$ be a poset, and let $U, V \subseteq \Pscr$ with $U \cup V = \Pscr$, such that:
    \begin{itemize}
        \item for every $U \ni u \leq v \in V$, the set $U \cap V \cap \{c \in \Pscr : u \leq c \leq v\}$ is connected;
        \item for every $V \ni v \leq u \in U$, the set $U \cap V \cap \{c \in \Pscr : v \leq c \leq u\}$ is connected.
    \end{itemize}
    Then, given any pair of functors $M : U \to \Ccal$ and $N : V \to \Ccal$ with $M|_{U \cap V} = N|_{U \cap V}$, there exists a unique functor $L : \Pscr \to \Ccal$ such that $L|_U = M$ and $L|_V = N$.
\end{lem}
\begin{proof}
    We define a functor $L$ on objects in the only possible way, that is, we let $L(x) = M(x)$ if $x \in U$, and $L(x) = N(x)$ if $x \in V$.
    To define the structure morphisms of $L$, we prove that there is only one possible choice.
    We state this formally as Claim 3, below.
    In order to state and prove this claim, we need a few definitions and intermediate results.
    \medskip

    \noindent\emph{Definitions.}
    An \emph{admissible pair} consists of an ordered pair of elements $(x, y)$ of $\Pscr$ such that $x \leq y$, and either $x,y \in U$ or $x,y \in V$.
    Given $x,y \in \Pscr$, an \emph{admissible path} between $x$ and $y$ consists of a finite ordered sequence $p = (p_1, \dots, p_k)$ of elements of $\Pscr$ such that $k\geq 2$, $x = p_1$, $y = p_k$, and for every $1 \leq i \leq k-1$ the pair $(p_i,p_{i+1})$ is admissible.
    In particular, an admissible pair is also an admissible path.
    Any admissible pair $(x,y) \in \Pscr$ induces a morphism $L(x,y) : L(x) \to L(y)$ defined as $\phi^M_{x,y}$ if $x,y \in U$ and as $\phi^N_{x,y}$ if $x,y \in V$.
    Then, any admissible path $p$ between $x$ and $y$ induces a morphism $L(p) : L(x) \to L(y)$, defined inductively as follows:
    \begin{itemize}
        \item If $p$ has length $2$, use $L(p)$ as defined above.
        \item If $p$ has length $3$ or more, let $L(p) = L(p_2, \dots, p_k) \circ L(p_1,p_2)$.
    \end{itemize}
    Two admissible paths $p$ and $q$ between $x$ and $y$ \emph{agree} if $L(p) = L(q)$ as morphisms $L(x) \to L(y)$.
    A standard argument by induction on the length of $p$ shows that, if $p = (p_1, \dots, p_k)$ is an admissible path and $2 \leq i \leq k-1$, then
    \begin{equation}
        \label{equation:associativity-concatenation}
        L(p) = L(p_i, \dots, p_k) \circ L(p_1, \dots, p_i). 
    \end{equation}

    Consider the function $f : \Pscr \to \{\ubf, \vbf, \ibf\}$ mapping the elements of $U\setminus V$ to $\ubf$, those of $U \setminus V$ to $\vbf$, and those of $U \cap V$ to $\ibf$.
    If $p = (p_1, \dots, p_k)$ is an admissible path, we let $f(p) = (f(p_1), \dots, f(p_k)) \subseteq \{\ubf, \vbf, \ibf\}$, as an ordered sequence.
    An ordered triple $(x, y, z) \subseteq \Pscr$ is \emph{reduced} if it is an admissible path and $(f(x), f(y), f(z))$ is equal, as an ordered sequence, to either $(\ubf, \ibf, \vbf)$ or to $(\vbf, \ibf, \ubf)$.
    An admissible path $p$ is \emph{reduced} if all of its consecutive triples are reduced.

    \medskip

    \noindent \emph{Claim 1. An admissible path $p$ is reduced if and only if it has length $2$ or it is a reduced triple.}

    It is sufficient to prove that $p$ cannot have length more than $3$.
    Consider $f(p)$ as a sequence of elements of $\{\ubf, \vbf, \ibf\}$.
    Since $p$ is reduced, a symbol $\ubf$ cannot follow or come before a symbol $\vbf$ in $f(p)$.
    This implies that either all elements of $f(p)$ in even positions are equal to $\ibf$, or all elements of $f(p)$ in odd positions are equal to $\ibf$.
    In either of these cases, $p$ cannot have length more than $3$ since it would otherwise contain a consecutive triple $(x,y,z)$ with $f(x) = f(z) = \ibf$, which would not be reduced.

    \medskip

    \noindent \emph{Claim 2. If an admissible path $(x,y,z)$ is not reduced as an ordered triple, then the pair $(x,z)$ is admissible.
    In that case, $L(x,y,z) = L(x,z)$, that is $(x,y,z)$ and $(x,z)$ agree.}

    Since $(x,y,z)$ is admissible, in $f(x,y,z)$ a symbol $\ubf$ cannot follow or come before a symbol $\vbf$ so either $f(y) = \ibf$ or $f(x) = f(z) = \ibf$.
    If $f(y) = \ibf$ and $(x,y,z)$ is not admissible, then $f(x) = f(z)$ so that $(x,z)$ is admissible and $L(x,y,z) = L(y,z) \circ L(x,y) = L(x,z)$, by \cref{equation:associativity-concatenation}, since in this case $x$, $y$, and $z$ are either all in $U$ or all in $V$.
    If $f(x) = f(z) = \ibf$, then $(x,z)$ is also admissible and 
    $L(x,y,z) = L(y,z) \circ L(x,y) = L(x,z)$, by \cref{equation:associativity-concatenation}, since also in this case $x$, $y$, and $z$ are either all in $U$ or all in $V$.
    \medskip
    
    \noindent \emph{Claim 3. 
    Let $x \leq y \in \Pscr$.
    For every admissible path $p$ between $x$ and $y$ there exists a reduced path $\overline{p}$ between $x$ and $y$ such that $p$ and $\overline{p}$ agree.}
    
    If $p$ is reduced, let $\overline{p} = p$.
    It is clear that we have $L(p) = L(\overline{p})$.
    If $p$ is not reduced, then we proceed inductively in the length of $p$, as follows.

    If $p = (p_1, \dots, p_k)$ is not reduced, let $1 \leq i \leq k-2$ be the smallest index such that $(p_i, p_{i+1}, p_{i+2})$ is not a reduced triple, and let $\overline{p} = \overline{p'}$, where
    \[
         p' = (p_j)_{1 \leq j \leq i} \; (p_j)_{i+2 \leq j \leq k}\, ,
    \]
    where juxtaposition of paths denotes concatenation of sequences.
    In other words, $p'$ is given by removing the $(i+1)$st component (that is $p_{i+1}$) from $p$.
    The path $p'$ is admissible since, if $(p_i, p_{i+1}, p_{i+2})$ is not reduced, then $(p_1, p_{i+2})$ is admissible by Claim 2.
    Moreover, we have
    \begin{align*}
        L(p') &= L(p_j)_{i+2 \leq j \leq k} \circ L(p_i, p_{i+2}) \circ L(p_j)_{1 \leq j \leq i}\\
            &= L(p_j)_{i+2 \leq j \leq k} \circ L(p_i, p_{i+1}, p_{i+2}) \circ L(p_j)_{1 \leq j \leq i}\\
            &= 
            L(p),
    \end{align*}
    using \cref{equation:associativity-concatenation} for the first and third equality, and Claim 2 for the second one.
    Finally, $L(p) = L(p') = L(\overline{p'}) = L(\overline{p})$, where the second equality holds by inductive hypothesis and the third equality holds by definition.

    \medskip

    \noindent \emph{Claim 4.
    Let $x \leq y \in \Pscr$.
    There exists at least one admissible path between $x$ and $y$, and all admissible paths between $x$ and $y$ agree.}

    By Claim 3, it is enough to consider reduced paths.
    If $(x,y)$ is an admissible pair, then there is only one reduced path between them.
    So let us assume that $(x,y)$ is not an admissible pair, and without loss of generality, we may assume that $x \in U \setminus V$ and $y \in V \setminus U$.
    Then, by Claim 1, the reduced paths between $x$ and $y$ are of the form $(x,c,z)$ with $c \in U \cap V$.
    We say that two reduced paths $(x,c,z)$ and $(x,c',z)$ are \emph{contiguous} if either $c \leq c'$ or $c' \leq c$.
    We now show that any pair of such contiguous reduced paths agree.
    Without loss of generality, assume $c \leq c'$, so that we have
    \[
        L(x,c,y) = L(c,y) \circ L(x,c) = L(c',y) \circ L(c,c') \circ L(x,c) = L(c',y) \circ L(x,c'),
    \]
    since all commutativity relations used occur in either $U$ or in $V$.
    To see that all reduced paths between $x$ and $y$ agree, and that there exists at least one such reduced path, it is thus sufficient to show that the transitive closure of the contiguity relation between reduced paths between $x$ and $y$ has exactly one equivalence class, which follows directly from the assumption that the set $U \cap V \cap \{c \in \Pscr : u \leq c \leq v\}$ is connected.

    \medskip

    We now use Claim 4 to conclude the proof of the result.
    Given $x \leq y \in \Pscr$, we define $\phi^L_{x,y} = L(p)$, where $p$ is any admissible path from $x$ to $y$.
    This is uniquely determined by Claim 4.
    By construction $\phi^L_{x,x}$ is the identity of $L(x)$, since $(x,x)$ is a reduced path.
    To see that $\phi^L_{y,z} \circ \phi^L_{x,y} = \phi^L_{x,z}$, note that the concatenation of an admissible path from $x$ to $y$ with an admissible path from $y$ to $z$ gives an admissible path from $x$ to $z$, so the equality follows from the fact that all admissible paths between $x$ and $z$ agree.
    It is clear that $L$ is the unique functor such that $L|_U = M$ and $L|_V = N$.
%
%
\end{proof}

\begin{lem}
    \label{lemma:gluing-indecomposable}
    Let $\Pscr$ be a poset, and let $U,V \subseteq \Pscr$ be such that $U \cup V = \Pscr$ and $U \cap V \neq \emptyset$.
    Assume given a persistence module $L : \Pscr \to \vect$ such that $L|_V$ is indecomposable, and such that $L|_U$ admits a decomposition $\bigoplus_{i \in I} X_i$ into indecomposables with $X_i|_{U \cap V} \neq 0$ for every $i \in I$.
    Then $L$ is indecomposable.
\end{lem}

For a related result, see \cite[Theorem~10.2]{chacholski-giunti-landi-tombari}.
\begin{proof}
    Since $L|_V$ is indecomposable, it is non-zero, and thus $L$ is non-zero.
    Now, assume that $L \cong S \oplus T$.
    On one hand, we have $L|_V \cong S|_V \oplus T|_V$.
    Since $L|_V$ is indecomposable, without loss of generality, it must be the case that $T|_V = 0$.
    In particular, $T|_{U\cap V} = 0$.
    On the other hand, $L|_U \cong S|_U \oplus T|_U$ so there exist disjoint subsets $J,J' \subseteq I$ with $I = J \cup J'$, $S|_U \cong \bigoplus_{i \in J} X_i$, and $T|_U \cong \bigoplus_{i \in J'} X_i$.
    Since all the indecomposable summands $X_i$ of $L|_U$ satisfy $X_i|_{U \cap V} \neq 0$ and $T|_{U \cap V} = 0$, we must have $J' = \emptyset$ and $T|_U = 0$.
    Thus, $T = 0$ and $L$ does not admit any non-trivial decompositions.
\end{proof}

\begin{dfn}
Let $\delta > 0$ and $n \geq 1 \in \Nbb$.
A \emph{hyper-rectangle} of $(\delta\Zbf)^n$ is any set of the form $S=[a_1, b_1] \times \cdots \times [a_n, b_n] \subseteq \delta\Zbf^n$, with $a_i \leq b_i \in \delta\Zbf \cup \{- \infty, +\infty\}$.
Given this, let
    \begin{align*}
        S^\uparrow &= [a_1, b_1 + \delta] \times \cdots \times [a_n, b_n + \delta] \\
        S^\downarrow &= [a_1-\delta, b_1] \times \cdots \times [a_n-\delta, b_n]\\
        \partial^\uparrow S &= S^\uparrow \setminus S\\
        \partial^\downarrow S &= S^\downarrow \setminus S\\
        \partial S &= \partial^\uparrow S \cup \partial^\downarrow S\\
        \overline{S} &= S \cup \partial S.
    \end{align*}
\end{dfn}
For an example, see \cref{figure:rectangle}.

\begin{figure}[h]
    \centering
    \includegraphics[width=0.5\linewidth]{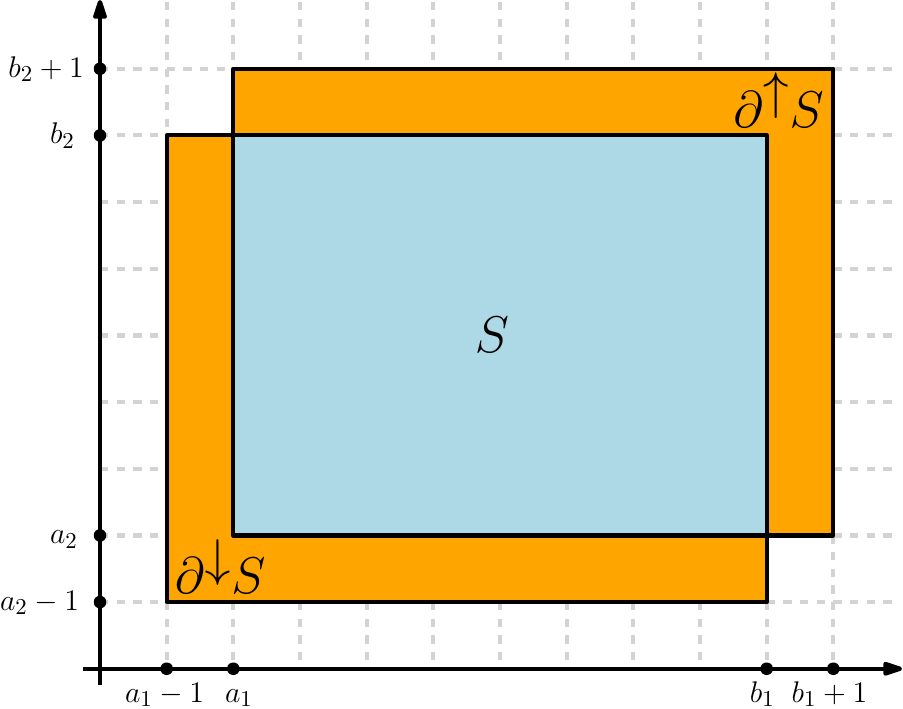}
    \caption{A hyper-rectangle $S = [a_1,b_1] \times [a_2, b_2] \subseteq \Zbf^2$.}
    \label{figure:rectangle}
\end{figure}

\begin{cor}
    \label{corollary:gluing-functors-rectangle}
    Let $\delta > 0$, let $n \geq 1 \in \Nbb$, let $S = [a_1, b_1] \times \cdots \times [a_n, b_n] \subseteq (\delta\Zbf)^n$ be a hyper-rectangle, and let $\Ccal$ be a category.
    Denote $U = \overline{S}$ and $V = (\delta\Zbf)^n \setminus S$.
    Given any pair of functors $M : U \to \Ccal$ and $N : V \to \Ccal$ with $M|_{U \cap V} = N|_{U \cap V}$, there exists a unique functor $L : (\delta \Zbf)^n \to \Ccal$ such that $L|_U = M$ and $L|_V = N$.
\end{cor}
\begin{proof}
    Without loss of generality, we may assume that $\delta = 1$.
    We use \cref{lemma:pushout-subposets}, so we need to prove that, given any hyper-rectangle $S$ as in the statement, we have that the set $U \cap V \cap \{c \in \Pscr : u \leq c \leq v\}$ is connected for every $U \ni u \leq v \in V$, and that the set $U \cap V \cap \{c \in \Pscr : v \leq c \leq u\}$ is connected for every $V \ni v \leq u \in U$.
    It is sufficient to prove the case $u \leq v$, since the other case is dual: to see this, consider the isomorphism between $\Zbf^n$ endowed with the opposite order and $\Zbf^n$ endowed with the standard order that maps $x$ to $-x$, and note that this isomorphism induces a bijection between the set of hyper-rectangles of $\Zbf^n$ and itself, so we can use this duality to reduce the case $v \leq u$ to the case $u \leq v$.

    So let $U \ni u \leq v \in V$.
    We will need the following claim, which follows directly from the fact that $U \cap V = \partial S$ and thus $\{c \in \Zbf^n : c \geq u\} \cap U \cap V = \{c \in \Zbf^n : c \geq u\} \cap \partial S$.

    \medskip
    \noindent \emph{Claim a. For every $c \geq u$, we have that $c \in U \cap V$ if and only if
    \begin{itemize}
        \item for every $1 \leq j \leq n$, $c_j \leq b_j + 1$,
        \item and there exists $1 \leq i \leq n$ such that $c_i = b_i + 1$.
    \end{itemize}
    }
    \medskip

    We start by proving that $\Xcal = U \cap V \cap \{c \in \Pscr : u \leq c \leq v\}$ is non-empty.
    If $v \in U \cap V$ or $u \in U \cap V$, this is clear, so assume that $v \in V \setminus U = \Zbf^n \setminus \overline{S}$ and $u \in U \setminus V = S$.
    Then, there must exist $1 \leq i \leq n$ such that $v_i > b_i + 1$. 
    Define $c \coloneqq u + (b_i + 1 - u_i)\ebf_i$, where, in words, what we are doing is to move closer to $v$, starting from $u$, and in the direction of the $i$th component, just enough to land in the boundary of $S$.
    It is clear that $u \leq c \leq v$, $c_i = b_i + 1$, and $c_j \leq b_j + 1$ for all $1 \leq j \leq n$.
    Thus, $c \in \Xcal$, by Claim \emph{a}.

    To conclude, we prove that $\Xcal$ is connected.
    Let $x,y \in \Xcal$.
    Then,
    \[
        x\vee y = \left(\max(x_1, y_1), \dots, \max(x_n, y_n)\right) \in \Xcal,
    \]
    since $u \leq x\vee y \leq v$, and $x\vee y \in \Xcal$ by Claim \emph{a}.
    Since $x \leq x \vee y \geq y$ and $\Xcal$ is non-empty, it follows that $\Xcal$ is also connected, as required.
\end{proof}

The next result lets us modify a module over a grid locally, only in a hyper-rectangle, and gives sufficient conditions for concluding that the modified module is indecomposable.

\begin{cor}
    \label{corollary:modify-on-rectangle}
    Let $\delta > 0$, let $n \geq 1 \in \Nbb$, let $S = [a_1, b_1] \times \cdots \times [a_n, b_n] \subseteq (\delta\Zbf)^n$ be a hyper-rectangle, and let $M : (\delta \Zbf)^n \to \vect$.
    Assume given $U \subseteq (\delta \Zbf)^n$ and $N : U \to \vect$ such that $\overline{S} \subseteq U$ and $M|_{U \setminus S} = N|_{U \setminus S}$.
    Denote $V = (\delta \Zbf)^n \setminus S$.
    \begin{enumerate}
        \item There exists a unique functor $L : (\delta \Zbf)^n \to \vect$ such that $L|_S = N|_S$ and $L|_V = M|_V$, which is finitely presentable if both $M$ and $N$ are.
        \item If $M|_V$ is indecomposable, and every indecomposable summand of $N$ is non-zero when restricted to $U \setminus S$, then $L$ is indecomposable.
    \end{enumerate}
\end{cor}
\begin{proof}
    We start by proving the first statement.
    Since $M|_{U \setminus S} = N|_{U \setminus S}$ and $\overline{S} \subseteq U$, it is enough to assume that $U = \overline{S}$.
    Then, the existence and uniqueness of $L$ follows directly from \cref{corollary:gluing-functors-rectangle}.
    If $M$ and $N$ are finitely presentable, then the module $L$ is also finitely presentable by \cref{lemma:fp-is-extension-finite-poset}.

    We now prove the second statement and do not assume that $U = \overline{S}$.
    If $M|_V$ is indecomposable and every indecomposable summand of $N$ is non-zero when restricted to $U \setminus S$, then $L|_V$ is indecomposable and every summand of $L|_U$ is non-zero when restricted to $U \cap V$, so \cref{lemma:gluing-indecomposable} implies that $L$ is indecomposable, as required.
\end{proof}

\subsection{Proof of the Tacking Lemma}
\label{section:tacking-lemma}

We first outline the proof strategy.
\cref{figure:diagram-proof-main-lemma} shows a diagrammatic description of the main steps in the two-parameter case (i.e.,~$n=2$).
We start by defining a persistence module $\Gsf$ on a finite grid which allows us to ``tack together'' indecomposable modules.
Now given modules $A$ and $B$ as in the statement, the tacking works in three steps.
First, we use \cref{lemma:add-1-dim-corner} to replace $A$ and $B$ by modules which have a ``thin corner'' (\cref{definition:1-dim-corner}); the purpose of this thin corner is that it allows us to attach a copy of $\Gsf$ to each of the modules, which we do in the next step.
Second, we use \cref{lemma:add-antenna-attachment} to replace the modules by modules that have an ``antenna attachment'' (\cref{definition:horizontal-antenna-attachment}) by attaching a copy of $\Gsf$ to each of the modules; the purpose of this antenna attachment is that it allows us to tack together the two modules with a third copy of $\Gsf$. 
Finally, using \cref{lemma:moving-attachment,lemma:actual-tacking}, we construct $M$ by tacking together the two modules.

\begin{figure}
    \centering
    \includegraphics[width=0.75\linewidth]{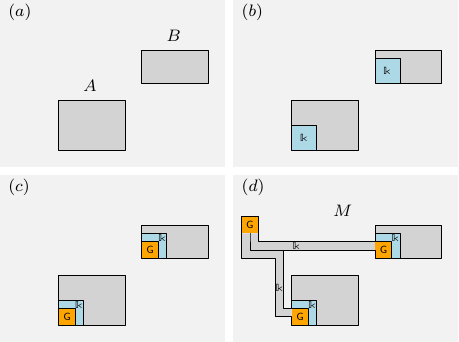}
    \caption{A schematic summary of the main steps in the proof of the tacking lemma.
    The transition from $(a)$ to $(b)$ corresponds to \cref{lemma:add-1-dim-corner}; the transition from $(b)$ to $(c)$ corresponds to \cref{lemma:add-antenna-attachment}; the transition from $(c)$ to $(d)$ corresponds to \cref{lemma:moving-attachment,lemma:actual-tacking}.}
    \label{figure:diagram-proof-main-lemma}
\end{figure}

In each of the steps in the proof of the tacking lemma, we modify modules in sufficiently small regions to guarantee that the modified modules are not too far from the original modules in the interleaving distance.
To make this formal, we use the following definitions and lemmas.

\begin{dfn}
    \label{definition:trivial-set}
    Let $\epsilon > 0 \in \Rbf$ and let $U \subseteq \Rbf^n$ be a set.
    We say that $U$ is \emph{$\epsilon$-trivial} if $U \cap (U+\epsilon) =\emptyset$, where $U + \epsilon = \{u + \epsilon : u \in U\}$.
\end{dfn}

The next lemma gives a simple way of checking that certain sets are trivial in the sense of \cref{definition:trivial-set}; its proof is straightforward.

\begin{lem}
    \label{lemma:extension-set-is-trivial}
    Let $k \in \Nbb$, $\epsilon > 0 \in \Rbf$, and $U \subseteq (\epsilon \Zbf)^n$.
    If $U \cap (U+k\epsilon) =\emptyset$, then, the extension $\widehat{U} \subseteq \Rbf^n$ of $U$ to $\Rbf^n$, defined as 
    \[
        \widehat{U} \coloneqq \{u + r : u \in U, r \in [0,\epsilon)^n\} \subseteq \Rbf^n,
    \]
    is $k\epsilon$-trivial.
\end{lem}

\begin{dfn}
	Let $A , B \colon \Rbf^n \to \vect$, and let $U \subseteq \Rbf^n$.
    We say that $A$ and $B$ \emph{agree outside $U$} if 
    $A|_{U^c} \cong B|_{U^c}$, where $U^c \subseteq \Rbf^n$ is the complement of $U$ seen as a full subposet of $\Rbf^n$.
\end{dfn}

\begin{lem}
    \label{lemma:local-changes}
	Let $A , B \colon \Rbf^n \to \vect$, let $\epsilon > 0$, and let $U \subseteq \Rbf^n$ be $\epsilon$-trivial.
    If $A$ and $B$ agree outside $U$, then they are $\epsilon$-interleaved.
\end{lem}
\begin{proof}
Without loss of generality, we may assume that $A|_{U^c} = B|_{U^c}$.
An interleaving morphism $f \colon A \to B[\epsilon]$ is canonically given by 
\begin{align*}
f_{t} &= 
\begin{dcases}
\phi^A_{t,t+\epsilon} \colon A(t) \to B({t+\epsilon}) = A({t+\epsilon}) & t \in U , \\
\phi^B_{t,t+\epsilon} \colon A(t) = B(t) \to B({t+\epsilon}) & t \notin U . \\
\end{dcases} 
\end{align*}
An interleaving inverse $g \colon B \to A[\epsilon]$ is given similarly as
\begin{align*}
g_{t} &= 
\begin{dcases}
\phi^B_{t,t+\epsilon} \colon B(t) \to A({t+\epsilon}) = B({t+\epsilon}) & t \in U  , \\
\phi^A_{t,t+\epsilon} \colon B(t) = A(t) \to A({t+\epsilon}) & t \notin U . \\
\end{dcases} 
\end{align*}
It follows immediately from the definition that $g[\epsilon] \circ f = \eta^A_{2\epsilon}$ and $f[\epsilon] \circ g = \eta^B_{2\epsilon}$, proving the claim.
\end{proof}

We now define the module $\Gsf$, which is the main building block for tacking modules together.
    Let $\Pscr = \{0, 1, 2, 3, 4\}^2$ and define $\Gsf : \Pscr \to \vect$ as follows:
    \[ \footnotesize
        \begin{tikzpicture}
            \matrix (m) [matrix of math nodes,row sep=3em,column sep=3em,minimum width=2em,nodes={text height=1.75ex,text depth=0.25ex}]
            {  \kbb & \kbb   & \kbb   & \kbb   & \kbb \\
               \kbb & \kbb^2 & \kbb^2 & \kbb^2 & \kbb \\
               0    & \kbb   & \kbb^2 & \kbb^2 & \kbb \\
               0    & 0      & \kbb   & \kbb^2 & \kbb \\
               0    & 0      & 0      & \kbb   & \kbb \\};
            \path[line width=0.5pt, -{>[width=6pt]}]
            (m-1-1) edge [-,double equal sign distance] (m-1-2)
            (m-1-2) edge [-,double equal sign distance] (m-1-3)
            (m-1-3) edge [-,double equal sign distance] (m-1-4)
            (m-1-4) edge [-,double equal sign distance] (m-1-5)

            (m-2-2) edge [-,double equal sign distance] (m-2-3)
            (m-2-3) edge [-,double equal sign distance] (m-2-4)

            (m-3-3) edge [-,double equal sign distance] (m-3-4)

            (m-5-5) edge [-,double equal sign distance] (m-4-5)
            (m-4-5) edge [-,double equal sign distance] (m-3-5)
            (m-3-5) edge [-,double equal sign distance] (m-2-5)
            (m-2-5) edge [-,double equal sign distance] (m-1-5)

            (m-4-4) edge [-,double equal sign distance] (m-3-4)
            (m-3-4) edge [-,double equal sign distance] (m-2-4)

            (m-3-3) edge [-,double equal sign distance] (m-2-3)

            (m-2-1) edge node [left] {\footnotesize $0$} (m-1-1)
            (m-2-1) edge node [above] {\footnotesize $\begin{pmatrix} 1\\ 1\end{pmatrix}$} (m-2-2)
            (m-2-2) edge node [right] {\footnotesize $(1,-1)$} (m-1-2)
            (m-2-3) edge node [right] {\footnotesize $(1,-1)$} (m-1-3)
            (m-2-4) edge node [right] {\footnotesize $(1,-1)$} (m-1-4)

            (m-5-4) edge node [above] {\footnotesize $0$} (m-5-5)
            (m-5-4) edge node [left] {\footnotesize $\begin{pmatrix} 1\\ 1\end{pmatrix}$} (m-4-4)
            (m-4-4) edge node [above] {\footnotesize $(1,-1)$} (m-4-5)
            (m-3-4) edge node [above] {\footnotesize $(1,-1)$} (m-3-5)
            (m-2-4) edge node [above] {\footnotesize $(1,-1)$} (m-2-5)

            (m-3-2) edge node [left] {\footnotesize $\begin{pmatrix} 1\\ 0\end{pmatrix}$} (m-2-2)
            (m-3-2) edge node [above] {\footnotesize $\begin{pmatrix} 1\\ 0\end{pmatrix}$} (m-3-3)

            (m-4-3) edge node [left] {\footnotesize $\begin{pmatrix} 0\\ 1\end{pmatrix}$} (m-3-3)
            (m-4-3) edge node [above] {\footnotesize $\begin{pmatrix} 0\\ 1\end{pmatrix}$} (m-4-4)

            (m-3-1) edge (m-2-1)
            (m-4-1) edge (m-3-1)
            (m-5-1) edge (m-4-1)

            (m-4-2) edge (m-3-2)
            (m-5-2) edge (m-4-2)

            (m-5-3) edge (m-4-3)

            (m-5-1) edge (m-5-2)
            (m-5-2) edge (m-5-3)
            (m-5-3) edge (m-5-4)

            (m-4-1) edge (m-4-2)
            (m-4-2) edge (m-4-3)
        
            (m-3-1) edge (m-3-2)
            ;
        \end{tikzpicture}
    \]

\begin{restatable}{lem}{Gindecomposable}
    \label{lemma:G-is-indecomposable}
    The persistence module $\Gsf$ is indecomposable.
\end{restatable}
\begin{proof}
    Let $f : \Gsf \to \Gsf$ be an endomorphism.
    If $f(4,4) : \kbb \to \kbb$ is given by multiplication by $\alpha$, then, for all $0 \leq j \leq 4$, we have that $f(4,j) : \kbb \to \kbb$ and $f(j,4) : \kbb \to \kbb$ must also be given by multiplication by $\alpha$.

    Assume that $f(1,2) : \kbb \to \kbb$ is given by multiplication by $\beta$ and that $f(2,1) : \kbb \to \kbb$ is given by multiplication by $\gamma$.
    Since the structure morphisms $\kbb = \Gsf(1,2) \to \Gsf(4,4) = \kbb$ is the identity, we must have $\beta = \alpha$.
    By an analogous argument, we have $\gamma = \alpha$.
    This also implies that for all $(j,k)$ with $\Gsf(j,k) = \kbb^2$ the endomorphism $f(j,k)$ is given by multiplication by $\alpha$.
    
    To conclude, note that, by the definition of the structure morphisms $\Gsf(0,3) \to \Gsf(1,3)$ and $\Gsf(3,0) \to \Gsf(3,1)$, we have that $f(0,3)$ and $f(3,0)$ are also given by multiplication by $\alpha$.
    Thus, $\End(\Gsf) \cong \End(\Gsf(4,4)) \cong \kbb$, and thus $\Gsf$ is indecomposable, by \cref{lemma:indecomposable-local-ring}.
\end{proof}

The crucial property of $\Gsf$ of importance to us, other than it being indecomposable, is the fact that there are two pairs of adjacent copies of the ground field $\kbb$ connected by a zero map, on the bottom right at $(0,3)$ and $(0,4)$, and on the top left at $(3,0)$ and $(4,0)$.

\begin{dfn}
    \label{definition:1-dim-corner}
    Let $A : \Rbf^n \to \vect$ be finitely presentable, and let $\Pscr \subseteq \Rbf^n$ be a regular grid.
    We say that $A$ \define{has a thin corner over $\Pscr$} if it is a $\Pscr$-extension and there exists $r \in \Pscr$ such that $A(r) = \kbb$ and $A(s) = 0$ for every $s < r$.
\end{dfn}

See \cref{figure:diagram-proof-main-lemma}$(b)$ for a schematic depiction of two modules having a thin corner.

\begin{notation}
    To help readability, the proofs of \cref{lemma:add-1-dim-corner,lemma:add-antenna-attachment,lemma:moving-attachment,lemma:actual-tacking} are structured similarly: we start with one or more modules on $\Rbf^n$ (denoted $A,B$), by restriction we get modules over a grid (denoted $X,Y$), by modifying these restricted modules locally we get new modules over the grid (denoted $X',Y'$), which are then extended to get modules over $\Rbf^n$ (denoted $A',B'$).
\end{notation}

\begin{lem}
    \label{lemma:add-1-dim-corner}
    Let $\epsilon > 0$, and let $A : \Rbf^n \to \vect$ be indecomposable, finitely presentable, and an $(\epsilon\Zbf)^n$-extension.
    There exists $A' : \Rbf^n \to \vect$ indecomposable, finitely presentable, having a thin corner over $(\frac\epsilon2\Zbf)^n$, and such that $A$ and $A'$ are isomorphic outside some $\frac\epsilon2$-trivial set.
\end{lem}
\begin{proof}
    Since $A$ is finitely presentable, there exists $r \in (\epsilon\Zbf)^n$ such that $A(r) \neq 0$ and $A(s) = 0$ for all $s < r$.
    We will modify $A$ only on the $\frac\epsilon2$-trivial set:
    \[
        \left[r_1, r_1+\frac\epsilon2\right) \times \cdots \times \left[r_n, r_n+\frac\epsilon2\right).
    \]

    \medskip

    \noindent\emph{Discretize $A$.}
    Let $X : (\frac{\epsilon}{2}\Zbf)^n \to \vect$ denote the restriction of $A$.
    Around index $r$ (highlighted entry), the module $X$ restricts to any axis-aligned $2$-dimensional slice corresponding to coordinate indices $1 \leq i < j \leq n$ as follows:
    \[ \scriptsize
        \begin{tikzpicture}
            \matrix (m) [matrix of math nodes,row sep=3em,column sep=0.8em,minimum width=2em,nodes={text height=1.75ex,text depth=0.25ex}]
            {  X(r-\epsilon\ebf_i+\epsilon\ebf_j) & X(r-\frac{\epsilon}{2}\ebf_i+\epsilon\ebf_j) & X(r + \epsilon\ebf_j) 
                   & X(r+\frac{\epsilon}{2}\ebf_i+\epsilon\ebf_j)   & X(r+\epsilon\ebf_i+\epsilon\ebf_j) \\
               0 & 0 & X(r+\frac{\epsilon}{2}\ebf_j)
                   & X(r+\frac{\epsilon}{2}\ebf_i+\frac{\epsilon}{2}\ebf_j) & X(r+\epsilon\ebf_i+\frac{\epsilon}{2}\ebf_j) \\
               0    & 0      & |[fill=lightgray!25, rectangle, outer sep = 2pt, minimum size = 0]| X(r) & X(r+\frac{\epsilon}{2}\ebf_i) & X(r+\epsilon\ebf_i) \\
               0    & 0      & 0      & 0 & X(r+\epsilon\ebf_i-\frac{\epsilon}{2}\ebf_j) \\
               0    & 0      & 0      & 0   & X(r+\epsilon\ebf_i-\epsilon\ebf_j) \\};
            \path[line width=0.5pt, -{>[width=4pt]}]
            (m-1-1) edge node [above] {$\cong$} (m-1-2)
            (m-2-1) edge node [above] {$\cong$} (m-2-2)
            (m-3-1) edge node [above] {$\cong$} (m-3-2)
            (m-4-1) edge node [above] {$\cong$} (m-4-2)
            (m-5-1) edge node [above] {$\cong$} (m-5-2)

            (m-1-2) edge (m-1-3)
            (m-2-2) edge (m-2-3)
            (m-3-2) edge (m-3-3)
            (m-4-2) edge (m-4-3)
            (m-5-2) edge (m-5-3)

            (m-1-3) edge node [above] {$\cong$} (m-1-4)
            (m-2-3) edge node [above] {$\cong$} (m-2-4)
            (m-3-3) edge node [above] {$\cong$} (m-3-4)
            (m-4-3) edge node [above] {$\cong$} (m-4-4)
            (m-5-3) edge node [above] {$\cong$} (m-5-4)

            (m-1-4) edge (m-1-5)
            (m-2-4) edge (m-2-5)
            (m-3-4) edge (m-3-5)
            (m-4-4) edge (m-4-5)
            (m-5-4) edge (m-5-5)

            (m-5-1) edge node [left] {$\cong$} (m-4-1)
            (m-5-2) edge node [left] {$\cong$} (m-4-2)
            (m-5-3) edge node [left] {$\cong$} (m-4-3)
            (m-5-4) edge node [left] {$\cong$} (m-4-4)
            (m-5-5) edge node [left] {$\cong$} (m-4-5)

            (m-4-1) edge (m-3-1)
            (m-4-2) edge (m-3-2)
            (m-4-3) edge (m-3-3)
            (m-4-4) edge (m-3-4)
            (m-4-5) edge (m-3-5)

            (m-3-1) edge node [left] {$\cong$} (m-2-1)
            (m-3-2) edge node [left] {$\cong$} (m-2-2)
            (m-3-3) edge node [left] {$\cong$} (m-2-3)
            (m-3-4) edge node [left] {$\cong$} (m-2-4)
            (m-3-5) edge node [left] {$\cong$} (m-2-5)

            (m-2-1) edge (m-1-1)
            (m-2-2) edge (m-1-2)
            (m-2-3) edge (m-1-3)
            (m-2-4) edge (m-1-4)
            (m-2-5) edge (m-1-5)
            ;
        \end{tikzpicture}
    \]
    where the isomorphisms are due to the fact that $A$ is an $(\epsilon\Zbf)^n$-extension.

    \medskip

    \noindent\emph{Modify $X$.}
    We now define a persistence module $X'$, and then define $A' : \Rbf^n \to \vect$ to be the extension of $X'$.
    Let $X' : (\frac{\epsilon}{2}\Zbf)^n \to \vect$ coincide with $X$ at all places, except at $r$ where it takes the value $\kbb$.
    Let $\iota : \kbb \to X(r)$ be any non-zero morphism.
    For each $1 \leq i \leq n$, the structure morphism $X'(r) \to X'(r+\frac{\epsilon}{2}\ebf_i)$ is defined by composing the corresponding structure morphisms of $X$ with $\iota : \kbb \to X(r)$.
    Thus, around index $r$ (again highlighted), the module $X'$ restricted to any axis-aligned $2$-dimensional slice through $r$ in the direction of two coordinate indices $1 \leq i < j \leq n$ looks as follows:
    \[  \scriptsize
        \begin{tikzpicture}
            \matrix (m) [matrix of math nodes,row sep=3em,column sep=0.8em,minimum width=2em,nodes={text height=1.75ex,text depth=0.25ex}]
            {  X(r-\epsilon\ebf_i+\epsilon\ebf_j) & X(r-\frac{\epsilon}{2}\ebf_i+\epsilon\ebf_j) & X(r + \epsilon\ebf_j) 
                   & X(r+\frac{\epsilon}{2}\ebf_i+\epsilon\ebf_j)   & X(r+\epsilon\ebf_i+\epsilon\ebf_j) \\
               0 & 0 & X(r+\frac{\epsilon}{2}\ebf_j)
                   & X(r+\frac{\epsilon}{2}\ebf_i+\frac{\epsilon}{2}\ebf_j) & X(r+\epsilon\ebf_i+\frac{\epsilon}{2}\ebf_j) \\
               0    & 0      & |[fill=lightgray!25, rectangle, outer sep = 2pt, minimum size = 0]| \kbb & X(r+\frac{\epsilon}{2}\ebf_i) & X(r+\epsilon\ebf_i) \\
               0    & 0      & 0      & 0 & X(r+\epsilon\ebf_i-\frac{\epsilon}{2}\ebf_j) \\
               0    & 0      & 0      & 0   & X(r+\epsilon\ebf_i-\epsilon\ebf_j) \\};
            \path[line width=0.5pt, -{>[width=4pt]}]
            (m-1-1) edge (m-1-2)
            (m-2-1) edge (m-2-2)
            (m-3-1) edge (m-3-2)
            (m-4-1) edge (m-4-2)
            (m-5-1) edge (m-5-2)

            (m-1-2) edge (m-1-3)
            (m-2-2) edge (m-2-3)
            (m-3-2) edge (m-3-3)
            (m-4-2) edge (m-4-3)
            (m-5-2) edge (m-5-3)

            (m-1-3) edge (m-1-4)
            (m-2-3) edge (m-2-4)
            (m-3-3) edge (m-3-4)
            (m-4-3) edge (m-4-4)
            (m-5-3) edge (m-5-4)

            (m-1-4) edge (m-1-5)
            (m-2-4) edge (m-2-5)
            (m-3-4) edge (m-3-5)
            (m-4-4) edge (m-4-5)
            (m-5-4) edge (m-5-5)

            (m-5-1) edge (m-4-1)
            (m-5-2) edge (m-4-2)
            (m-5-3) edge (m-4-3)
            (m-5-4) edge (m-4-4)
            (m-5-5) edge (m-4-5)

            (m-4-1) edge (m-3-1)
            (m-4-2) edge (m-3-2)
            (m-4-3) edge (m-3-3)
            (m-4-4) edge (m-3-4)
            (m-4-5) edge (m-3-5)

            (m-3-1) edge (m-2-1)
            (m-3-2) edge (m-2-2)
            (m-3-3) edge (m-2-3)
            (m-3-4) edge (m-2-4)
            (m-3-5) edge (m-2-5)

            (m-2-1) edge (m-1-1)
            (m-2-2) edge (m-1-2)
            (m-2-3) edge (m-1-3)
            (m-2-4) edge (m-1-4)
            (m-2-5) edge (m-1-5)
            ;
        \end{tikzpicture}
    \]
    To see that the module $X'$ is well-defined, and that its extension $A'$ is finitely presentable, we use \cref{corollary:modify-on-rectangle}(1) with $\delta = \epsilon/2$, hyper-rectangle $S = \{r\}$, $M = X$, $U = \overline{S}$, and $N = X'|_U$, which is described above, and easily seen to be well-defined.

    \medskip

    \noindent\emph{The module $X'$ is indecomposable.}
    In order to see that $X'$ (and thus $A'$) is indecomposable, we use \cref{corollary:modify-on-rectangle}(2).
    It is clear that all indecomposable summands of $N$ are non-zero when restricted to $U \setminus S$ since the only element of $U$ not in $U \setminus S$ is $r$
    and the structure morphism $N(r) \to N(r+\frac{\epsilon}{2}\esf_1)$ is non-zero.
    So it remains to show that $X|_{(\frac{\epsilon}{2}\Zbf)^n \setminus S}$ is indecomposable.
    Consider the composite isomorphism of endomorphism rings
    \[
        \End\left(X|_{(\frac{\epsilon}{2}\Zbf)^n \setminus S}\right)
        \cong
        \End\left(X|_{\left(\epsilon\Zbf+\frac{\epsilon}{2}\right)^n}\right)
        \cong
        \End\left(A\right).
    \]
    The first isomorphism is due to the fact that
    the value of $X$ at any element of $(\frac{\epsilon}{2}\Zbf)^n \setminus S$ is isomorphic, through structure morphisms of $X$, to the value of $X$ at an element of $(\epsilon\Zbf + \frac{\epsilon}{2})^n$, and vice versa.
    The second isomorphism is because $X$ is given by restricting $A$ which is an $(\epsilon \Zbf)^n$-extension.
    Since $A$ is indecomposable, it follows that $X|_{(\frac{\epsilon}{2}\Zbf)^n \setminus S}$ is indecomposable, by \cref{lemma:indecomposable-local-ring}.
%

    \medskip

    To conclude, note that $A'$ has a thin corner over $(\frac\epsilon2\Zbf)^n$, by construction.
\end{proof}

\begin{dfn}
    \label{definition:horizontal-antenna-attachment}
    Let $A : \Rbf^n \to \vect$ be finitely presentable, let $\Pscr = (\epsilon\Zbf)^n$ be a regular grid, and let $1 \leq i \leq n$.
    We say that $A$ \define{has an axis-$i$-aligned antenna attachment over $\Pscr$ at $r \in \Pscr$} if $A$ is a $\Pscr$-extension, we have $A(r) = \kbb$ and $A(r-k\epsilon\ebf_i) = 0$ for all $k \geq 1 \in \Nbb$, and the structure morphism $A(r) \to A(r+\epsilon\ebf_j)$ is zero for every $1 \leq j \leq n$ with $j \neq i$.
\end{dfn}

\begin{exm}
    The extension $\widehat\Gsf: \Rbf^2 \to \vect$ has an axis-$1$-aligned antenna attachment over $\Zbf^2$ at $(0,3)$, as well as an axis-$2$-aligned antenna attachment over $\Zbf^2$ at $(3,0)$.
\end{exm}

\begin{lem}
    \label{lemma:add-antenna-attachment}
    Let $\epsilon > 0$, and let $A : \Rbf^n \to \vect$ be indecomposable, finitely presentable, and having a thin corner over $(\epsilon\Zbf)^n$.
    Let $\Qscr = (\frac\epsilon5\Zbf)^n$.
    There exists $A' : \Rbf^n \to \vect$ indecomposable, finitely presentable, having an axis-$1$-aligned antenna attachment over $\Qscr$, and such that $A$ and $A'$ are isomorphic outside an $\epsilon$-trivial set.
\end{lem}
\begin{proof}
    Since $A$ has a thin corner over $(\epsilon \Zbf)^n$, there exists $r \in (\epsilon \Zbf)^n$ such that $A(r) = \kbb$ and $A(s) = 0$ for all $s < r$.
    We will modify $A$ only on the $\epsilon$-trivial set
    \[
        \left[r_1, r_1+\frac{4\epsilon}{5}\right) \times \cdots \times \left[r_n, r_n+\frac{4\epsilon}{5}\right).
    \]

    \medskip

    \noindent\emph{Discretize $A$.}
    We make use of the regular grid $(\frac\epsilon5\Zbf)^n$.
    By restricting $A$ to $(\frac\epsilon5\Zbf)^n$, we get a module $X : (\frac\epsilon5\Zbf)^n \to \vect$ that, when restricted to the finite grid
    \[
        \Tscr \coloneqq \left\{-\frac{\epsilon}{5}, 0, \frac{\epsilon}{5}, \frac{2\epsilon}{5}, \frac{3\epsilon}{5}, \frac{4\epsilon}{5}\right\}^n + r
    \]
    and to any axis-aligned $2$-dimensional slice through $r$ in the direction of two coordinate indices $1 \leq i < j \leq n$
    looks as follows (index $r$ highlighted):
    \[\footnotesize
        \begin{tikzpicture}
            \matrix (m) [matrix of math nodes,row sep=3em,column sep=3em,minimum width=2em,nodes={text height=1.75ex,text depth=0.25ex}]
            { 0 & \kbb & \kbb   & \kbb   & \kbb   & \kbb \\
              0 & \kbb & \kbb & \kbb & \kbb & \kbb \\
              0 &  \kbb    & \kbb   & \kbb & \kbb & \kbb \\
             0 &   \kbb    & \kbb      & \kbb   & \kbb & \kbb \\
             0 &  |[fill=lightgray!25, rectangle, outer sep = 2pt, minimum size = 0]| \kbb    & \kbb      & \kbb      & \kbb   & \kbb \\
             0 &      0    & 0         & 0         & 0      &  0 \\};
            \path[line width=0.5pt, -{>[width=6pt]}]
            (m-1-2) edge [-,double equal sign distance] (m-1-3)
            (m-1-3) edge [-,double equal sign distance] (m-1-4)
            (m-1-4) edge [-,double equal sign distance] (m-1-5)
            (m-1-5) edge [-,double equal sign distance] (m-1-6)
            (m-2-3) edge [-,double equal sign distance] (m-2-4)
            (m-2-4) edge [-,double equal sign distance] (m-2-5)
            (m-3-4) edge [-,double equal sign distance] (m-3-5)
            (m-5-6) edge [-,double equal sign distance] (m-4-6)
            (m-4-6) edge [-,double equal sign distance] (m-3-6)
            (m-3-6) edge [-,double equal sign distance] (m-2-6)
            (m-2-6) edge [-,double equal sign distance] (m-1-6)
            (m-4-5) edge [-,double equal sign distance] (m-3-5)
            (m-3-5) edge [-,double equal sign distance] (m-2-5)
            (m-3-4) edge [-,double equal sign distance] (m-2-4)
            (m-2-2) edge [-,double equal sign distance] (m-1-2)
            (m-2-2) edge [-,double equal sign distance] (m-2-3)
            (m-2-3) edge [-,double equal sign distance] (m-1-3)
            (m-2-4) edge [-,double equal sign distance] (m-1-4)
            (m-2-5) edge [-,double equal sign distance] (m-1-5)
            (m-5-5) edge [-,double equal sign distance] (m-5-6)
            (m-5-5) edge [-,double equal sign distance] (m-4-5)
            (m-4-5) edge [-,double equal sign distance] (m-4-6)
            (m-3-5) edge [-,double equal sign distance] (m-3-6)
            (m-2-5) edge [-,double equal sign distance] (m-2-6)
            (m-3-3) edge [-,double equal sign distance] (m-2-3)
            (m-3-3) edge [-,double equal sign distance] (m-3-4)
            (m-4-4) edge [-,double equal sign distance] (m-3-4)
            (m-4-4) edge [-,double equal sign distance] (m-4-5)
            (m-3-2) edge [-,double equal sign distance] (m-2-2)
            (m-4-2) edge [-,double equal sign distance] (m-3-2)
            (m-5-2) edge [-,double equal sign distance] (m-4-2)
            (m-4-3) edge [-,double equal sign distance] (m-3-3)
            (m-5-3) edge [-,double equal sign distance] (m-4-3)
            (m-5-4) edge [-,double equal sign distance] (m-4-4)
            (m-5-2) edge [-,double equal sign distance] (m-5-3)
            (m-5-3) edge [-,double equal sign distance] (m-5-4)
            (m-5-4) edge [-,double equal sign distance] (m-5-5)
            (m-4-2) edge [-,double equal sign distance] (m-4-3)
            (m-4-3) edge [-,double equal sign distance] (m-4-4)
            (m-3-2) edge [-,double equal sign distance] (m-3-3)

            (m-1-1) edge (m-1-2)
            (m-2-1) edge (m-2-2)
            (m-3-1) edge (m-3-2)
            (m-4-1) edge (m-4-2)
            (m-5-1) edge (m-5-2)
            (m-6-1) edge (m-6-2)

            (m-6-2) edge (m-6-3)
            (m-6-3) edge (m-6-4)
            (m-6-4) edge (m-6-5)
            (m-6-5) edge (m-6-6)

            (m-2-1) edge (m-1-1)
            (m-3-1) edge (m-2-1)
            (m-4-1) edge (m-3-1)
            (m-5-1) edge (m-4-1)
            (m-6-1) edge (m-5-1)

            (m-6-2) edge (m-5-2)
            (m-6-3) edge (m-5-3)
            (m-6-4) edge (m-5-4)
            (m-6-5) edge (m-5-5)
            (m-6-6) edge (m-5-6)
            ;
        \end{tikzpicture}
    \]

    \medskip

    \noindent\emph{Modify $X$.}
    We now define a persistence module $X' : (\frac\epsilon5\Zbf)^n \to \vect$ and define $A' : \Rbf^n \to \vect$ to be the extension of $X'$.
    Let $X' : (\frac\epsilon5\Zbf)^n \to \vect$ coincide with $X$ at all places, except at the subgrid
    \[
        \Sscr \coloneqq \left\{-\frac{\epsilon}{5}, 0, \frac{\epsilon}{5}, \frac{2\epsilon}{5}, \frac{3\epsilon}{5}, \frac{4\epsilon}{5}\right\}^2 \times \left\{0\right\}^{n-2} + r \subseteq \Tscr,
    \]
    where we use a copy of the module $\Gsf$ as follows (index $r$ highlighted):
    \[ \footnotesize
        \begin{tikzpicture}
            \matrix (m) [matrix of math nodes,row sep=3em,column sep=3em,minimum width=2em,nodes={text height=1.75ex,text depth=0.25ex}]
            { 0 & \kbb & \kbb   & \kbb   & \kbb   & \kbb \\
              0 & \kbb & \kbb^2 & \kbb^2 & \kbb^2 & \kbb \\
              0 & 0    & \kbb   & \kbb^2 & \kbb^2 & \kbb \\
              0 & 0    & 0      & \kbb   & \kbb^2 & \kbb \\
              0 & |[fill=lightgray!25, rectangle, outer sep = 2pt, minimum size = 0]|0    & 0      & 0      & \kbb   & \kbb \\
              0 & 0    & 0      & 0      & 0      & 0    \\};
            \path[line width=0.5pt, -{>[width=6pt]}]
            (m-1-2) edge [-,double equal sign distance] (m-1-3)
            (m-1-3) edge [-,double equal sign distance] (m-1-4)
            (m-1-4) edge [-,double equal sign distance] (m-1-5)
            (m-1-5) edge [-,double equal sign distance] (m-1-6)

            (m-2-3) edge [-,double equal sign distance] (m-2-4)
            (m-2-4) edge [-,double equal sign distance] (m-2-5)

            (m-3-4) edge [-,double equal sign distance] (m-3-5)

            (m-5-6) edge [-,double equal sign distance] (m-4-6)
            (m-4-6) edge [-,double equal sign distance] (m-3-6)
            (m-3-6) edge [-,double equal sign distance] (m-2-6)
            (m-2-6) edge [-,double equal sign distance] (m-1-6)

            (m-4-5) edge [-,double equal sign distance] (m-3-5)
            (m-3-5) edge [-,double equal sign distance] (m-2-5)

            (m-3-4) edge [-,double equal sign distance] (m-2-4)

            (m-2-2) edge node [left] {\footnotesize $0$} (m-1-2)
            (m-2-2) edge node [above] {\footnotesize $\begin{pmatrix} 1\\ 1\end{pmatrix}$} (m-2-3)
            (m-2-3) edge node [right] {\footnotesize $(1,-1)$} (m-1-3)
            (m-2-4) edge node [right] {\footnotesize $(1,-1)$} (m-1-4)
            (m-2-5) edge node [right] {\footnotesize $(1,-1)$} (m-1-5)

            (m-5-5) edge node [above] {\footnotesize $0$} (m-5-6)
            (m-5-5) edge node [left] {\footnotesize $\begin{pmatrix} 1\\ 1\end{pmatrix}$} (m-4-5)
            (m-4-5) edge node [above] {\footnotesize $(1,-1)$} (m-4-6)
            (m-3-5) edge node [above] {\footnotesize $(1,-1)$} (m-3-6)
            (m-2-5) edge node [above] {\footnotesize $(1,-1)$} (m-2-6)

            (m-3-3) edge node [left] {\footnotesize $\begin{pmatrix} 1\\ 0\end{pmatrix}$} (m-2-3)
            (m-3-3) edge node [above] {\footnotesize $\begin{pmatrix} 1\\ 0\end{pmatrix}$} (m-3-4)

            (m-4-4) edge node [left] {\footnotesize $\begin{pmatrix} 0\\ 1\end{pmatrix}$} (m-3-4)
            (m-4-4) edge node [above] {\footnotesize $\begin{pmatrix} 0\\ 1\end{pmatrix}$} (m-4-5)

            (m-3-2) edge (m-2-2)
            (m-4-2) edge (m-3-2)
            (m-5-2) edge (m-4-2)

            (m-4-3) edge (m-3-3)
            (m-5-3) edge (m-4-3)

            (m-5-4) edge (m-4-4)

            (m-5-2) edge (m-5-3)
            (m-5-3) edge (m-5-4)
            (m-5-4) edge (m-5-5)

            (m-4-2) edge (m-4-3)
            (m-4-3) edge (m-4-4)
        
            (m-3-2) edge (m-3-3)

            (m-1-1) edge (m-1-2)
            (m-2-1) edge (m-2-2)
            (m-3-1) edge (m-3-2)
            (m-4-1) edge (m-4-2)
            (m-5-1) edge (m-5-2)
            (m-6-1) edge (m-6-2)

            (m-6-2) edge (m-6-3)
            (m-6-3) edge (m-6-4)
            (m-6-4) edge (m-6-5)
            (m-6-5) edge (m-6-6)

            (m-2-1) edge (m-1-1)
            (m-3-1) edge (m-2-1)
            (m-4-1) edge (m-3-1)
            (m-5-1) edge (m-4-1)
            (m-6-1) edge (m-5-1)

            (m-6-2) edge (m-5-2)
            (m-6-3) edge (m-5-3)
            (m-6-4) edge (m-5-4)
            (m-6-5) edge (m-5-5)
            (m-6-6) edge (m-5-6)
            ;
            ;
        \end{tikzpicture}
    \]
    The structure morphisms of $X'$ restricted to $\Tscr$ are chosen according to the conditions for functoriality of $X'$.
    More specifically, if $s \in \Sscr$ corresponds to $(x,y) \in \{0,1,2,3,4\}^2$, $t \in \Tscr \setminus \Sscr$, and $s \leq t$, then the morphism $X'(s) \to X'(t)$ is taken to be $X'(s) = \Gsf(x,y) \to \Gsf(4,4) = \kbb = X'(t)$.
    Note that all indices $t \in \Tscr \setminus \Sscr$ are connected via isomorphisms to the indices in the upper right corner of $\Sscr$.

    To see that the module $X'$ is well-defined, and that its extension $A'$ is finitely presentable, we use \cref{corollary:modify-on-rectangle}(1) with $\delta = \epsilon/5$, hyper-rectangle
    \[
        S = \left\{0, \frac{\epsilon}{5}, \frac{2\epsilon}{5}, \frac{3\epsilon}{5}\right\}^2 \times \left\{0\right\}^{n-2} + r,
    \]
    $M = X$, $U = \Sscr$, and $N = X'|_U$, which is described above, and easily seen to be well-defined.
    

    \medskip

    \noindent\emph{The module $X'$ is indecomposable.}
    In order to see that $X'$ (and thus $A'$) is indecomposable, we use \cref{corollary:modify-on-rectangle}(2).
    The module $N$ is indecomposable, by \cref{lemma:G-is-indecomposable}.
    So it remains to show that $X|_{(\frac{\epsilon}{5}\Zbf)^n \setminus S}$ is indecomposable. 
    Consider the composite isomorphism of endomorphism rings
    \[
        \End\left(X|_{(\frac{\epsilon}{5}\Zbf)^n \setminus S}\right)
        \cong
        \End\left(X|_{(\epsilon\Zbf)^n + \frac{4\epsilon}{5}}\right)
        \cong
        \End\left(A\right).
    \]
    The first isomorphism is due to the fact that
    for any element of $(\frac{\epsilon}{5}\Zbf)^n \setminus S$, there is an element of $(\epsilon\Zbf + \frac{4\epsilon}{5})^n$ such that the structure morphism of $X$ between these two elements is an isomorphisms, since $X$ is a restriction of $A$, which in turn is an extension of the $(\epsilon\Zbf)^n$ grid.
    The second isomorphism is because $A$ is an $(\epsilon \Zbf)^n$-extension.
    Since $A$ is indecomposable, it follows that $X|_{(\frac{\epsilon}{2}\Zbf)^n \setminus S}$ is indecomposable, by \cref{lemma:indecomposable-local-ring}.

    \medskip

    To conclude, note that $A'$ has an axis-$1$-aligned antenna attachment at $r+\frac{3\epsilon}{5}\ebf_2$, by construction.
\end{proof}

The following lemma allows us to move the coordinates of antenna attachments by a local modification on an $\epsilon$-trivial subset, resulting in a controlled perturbation in the interleaving distance.
Intuitively, the modification can be described as extending an antenna from the original attachment point, iteratively moving in coordinate directions alternating between moving backward in odd coordinates and forward in even coordinates.
We move in this way so that we can leave structure morphisms outside the $\epsilon$-trivial subset essentially untouched, without having to worry about the commutativity of the structure morphisms of the modified persistence module.

\begin{lem}
    \label{lemma:moving-attachment}
    Let $n \geq 2 \in \Nbb$, $\epsilon > 0$, and let $A : \Rbf^n \to \vect$ be indecomposable, finitely presentable, and having an axis-$1$-aligned antenna attachment over $(\epsilon\Zbf)^n$ at $r \in (\epsilon\Zbf)^n$.
    Let $s \in (\epsilon\Zbf)^n$ be such that $s_k < r_k$ for all $1 \leq k \leq n$ odd, $s_k > r_k$ for all $1 \leq k \leq n$ even,  and $A(t) = 0$ for $t \in (\epsilon\Zbf)^n$ with $t_1 \leq s_1$.
    Let $\ell = 1$ if $n$ is even, and $\ell = n$ if $n$ is odd.
    There exists $A' : \Rbf^n \to \vect$ indecomposable, finitely presentable, having an axis-$\ell$-aligned antenna attachment over $(\epsilon\Zbf)^n$ at $s$, and such that $A$ and $A'$ are isomorphic outside an $\epsilon$-trivial set.
\end{lem}

\begin{proof}
    Consider the following hyper-rectangles of $(\epsilon \Zbf)^n$
    \begin{align*}
        S_1 &= [s_1, r_1 - \epsilon]
                \times
                \left(\prod_{2 \leq j \leq n} \{r_j\} \right)\\
        S_2 &= \{s_1\}
                \times
                [r_2 + \epsilon, s_2]
                \times
                \left(\prod_{3 \leq j \leq n} \{r_j\}\right)\\
        &\;\;\vdots \\
        S_k &= \left(\prod_{1 \leq i \leq k-1} \{s_i\}\right)
            \times
            \begin{rcases}
                \begin{dcases}
                [s_k, r_k - \epsilon] & \text{ if $k$ is odd}\\
                [r_k + \epsilon, s_k] & \text{ if $k$ is even}
                \end{dcases}
              \end{rcases}
            \times
            \left(\prod_{k+1 \leq j \leq n} \{r_j\}\right)\\
        &\;\;\vdots
    \end{align*}
    We will modify $A$ only on the $\epsilon$-trivial set
    $T \coloneqq T_1 \cup \cdots \cup T_n \subseteq \Rbf^n$ where $T_k$ is defined to be the extension $\widehat{S_k}$ of $S_k$ to $\Rbf^n$, as in \cref{lemma:extension-set-is-trivial}; that is:
    \[
        T_k = \left(\prod_{1 \leq i \leq k-1} [s_i, s_i+\epsilon)\right)
            \times
            \begin{rcases}
                \begin{dcases}
                [s_k, r_k) & \text{ if $k$ is odd}\\
                [r_k + \epsilon, s_k + \epsilon) & \text{ if $k$ is even}
                \end{dcases}
              \end{rcases}
            \times
            \left(\prod_{k+1 \leq j \leq n} [r_j, r_j + \epsilon)\right).
    \]
    Since $A$ has an axis-1-aligned antenna attachment over $(\epsilon \Zbf)^n$ at $r$, we have $A|_{T_1} = 0$; and since $A(t) = 0$ whenever $t_1 \leq s_1$, we have $A|_{T_2 \cup \dots \cup T_n} = 0$.
    It follows that $A|_T = 0$.
    The idea is to let $A'$ coincide with $A$ everywhere, except on $T$, where it takes the value $\kbb$, with the new structure morphisms being the identity whenever possible.
    To make this precise, we proceed as follows.

    \medskip

    \noindent\emph{Discretize $A$.}
    Let $X : (\epsilon \Zbf)^n \to \vect$ denote the restriction of $A$ to $(\epsilon \Zbf)^n$.

    \medskip

    \noindent\emph{Modify $X$.}
    We now construct a sequence of indecomposable persistence modules $X_0, \dots, X_n : (\epsilon \Zbf)^n \to \vect$ inductively, starting with $X_0 = X$, and set $X' = X_n$.

    Let $X_{k}$ coincide with $X_{k-1}$ everywhere except on $S_k$, where we set it to be constantly $\kbb$.
    To see that the functor $X_k$ is well-defined and indecomposable, and that its extension is finitely presentable, we use \cref{corollary:modify-on-rectangle} with $\delta = \epsilon$, $S = S_k$, and $U = \overline{S}$.
    The details are similar to those in the proofs of \cref{lemma:add-1-dim-corner,lemma:add-antenna-attachment}, although simpler since we are replacing the zero vector spaces $X_{k-1}|_{S_k} = 0$ by $\kbb$ in $X_k|_{S_k}$.
    Let us give the details of the case $k=1$, since this is the case in which the assumption that $A$ has an antenna attachment at $r$ comes in.

    In the case $k=1$, we have $S = S_1$, and $N : \overline{S} \to \vect$ is defined to agree with $X_0 = X$ except on $S_1$ where it is defined to be constantly $\kbb$.
    The structure morphisms of $N$ are set to be the identity between elements of $S \cup \{r\}$ and zero otherwise, using the fact that for every $s \in S \cup \{r\}$ and $s' \in \partial S \setminus \{r\}$, the structure morphisms $X(s) \to X(s')$ is zero, since $X(s) = 0$ when $s \in S$, and $X(s) \to X(s')$ is zero when $s = r$ by the assumption that $X$ has an axis-1-aligned antenna attachment at $r$.
    Then, $N$ and $X$ agree on $U \setminus S = \partial S$ because $A$ has an axis-$1$-aligned antenna attachment at $r$.
    This lets us apply \cref{corollary:modify-on-rectangle}(1).
    Satisfying the hypothesis of \cref{corollary:modify-on-rectangle}(2) in this case is straightforward.

    \medskip

    To conclude, let $A'$ be the extension of $X'$ to $\Rbf^n$.
    The module $A'$ is finitely presentable and indecomposable.

    \medskip

    \noindent\emph{The module $A'$ has an antenna attachment.}
    The module $A'$ has an axis-$\ell$-aligned antenna attachment on $(\epsilon\Zbf)^n$ at
    \[
        s \in S_n = \left(\prod_{1 \leq i \leq n-1} \{s_i\}\right)
            \times
            \begin{rcases}
                \begin{dcases}
                [s_n, r_n - \epsilon] & \text{ if $n$ is odd}\\
                [r_n + \epsilon, s_n] & \text{ if $n$ is even}
                \end{dcases}
              \end{rcases}.
    \]
    To see this, recall that $\ell = 1$ if $n$ is even and $\ell = n$ if $n$ is odd.

    In the $n$ even case, the structure morphism $X'(s) \to X'(s + \epsilon\esf_i)$ is zero for all $1 \leq i \leq n$ because this is true for $X$ and $\{s + \epsilon\esf_i : 1 \leq i \leq n \}$ does not intersect $S_1 \cup \dots \cup S_n$.
    Also, $X'(s-k\epsilon\esf_1) = 0$ for all $k > 1 \in \Nbb$, since this is true for $X$ and $\{s-k\epsilon\esf_1 : k > 1 \in \Nbb\}$ does not intersect $S_1 \cup \dots \cup S_n$.

    In the $n$ odd case, the structure morphism $X'(s) \to X'(s + \epsilon\esf_i)$ is zero for all $1 \leq i < n$ because this is true for $X$ and $\{s + \epsilon\esf_i : 1 \leq i < n \}$ does not intersect $S_1 \cup \dots \cup S_n$.
    Also, $X'(s-k\epsilon\esf_n) = 0$ for all $k > 1 \in \Nbb$, since this is true for $X$ and $\{s-k\epsilon\esf_n : k > 1 \in \Nbb\}$ does not intersect $S_1 \cup \dots \cup S_n$.
\end{proof}


\begin{lem}
    \label{lemma:actual-tacking}
    Let $\ell \neq \ell' \in \{1, \dots, n\}$, let $\epsilon > 0$, and let $A, B : \Rbf^n \to \vect$ be indecomposable, finitely presentable, and having an axis-$\ell$-aligned antenna attachment over $(\epsilon\Zbf)^n$ and at $r \in (\epsilon\Zbf)^n$ and $r-\epsilon\ebf_{\ell'} \in (\epsilon\Zbf)^n$, respectively.
    There exists $M : \Rbf^n \to \vect$ indecomposable, finitely presentable, an $(\epsilon\Zbf)^n$-extension, and such that $M$ and $A\oplus B$ are isomorphic outside a $(5\epsilon)$-trivial set.
\end{lem}

\begin{figure}
    \centering
    \includegraphics[width=0.8\linewidth]{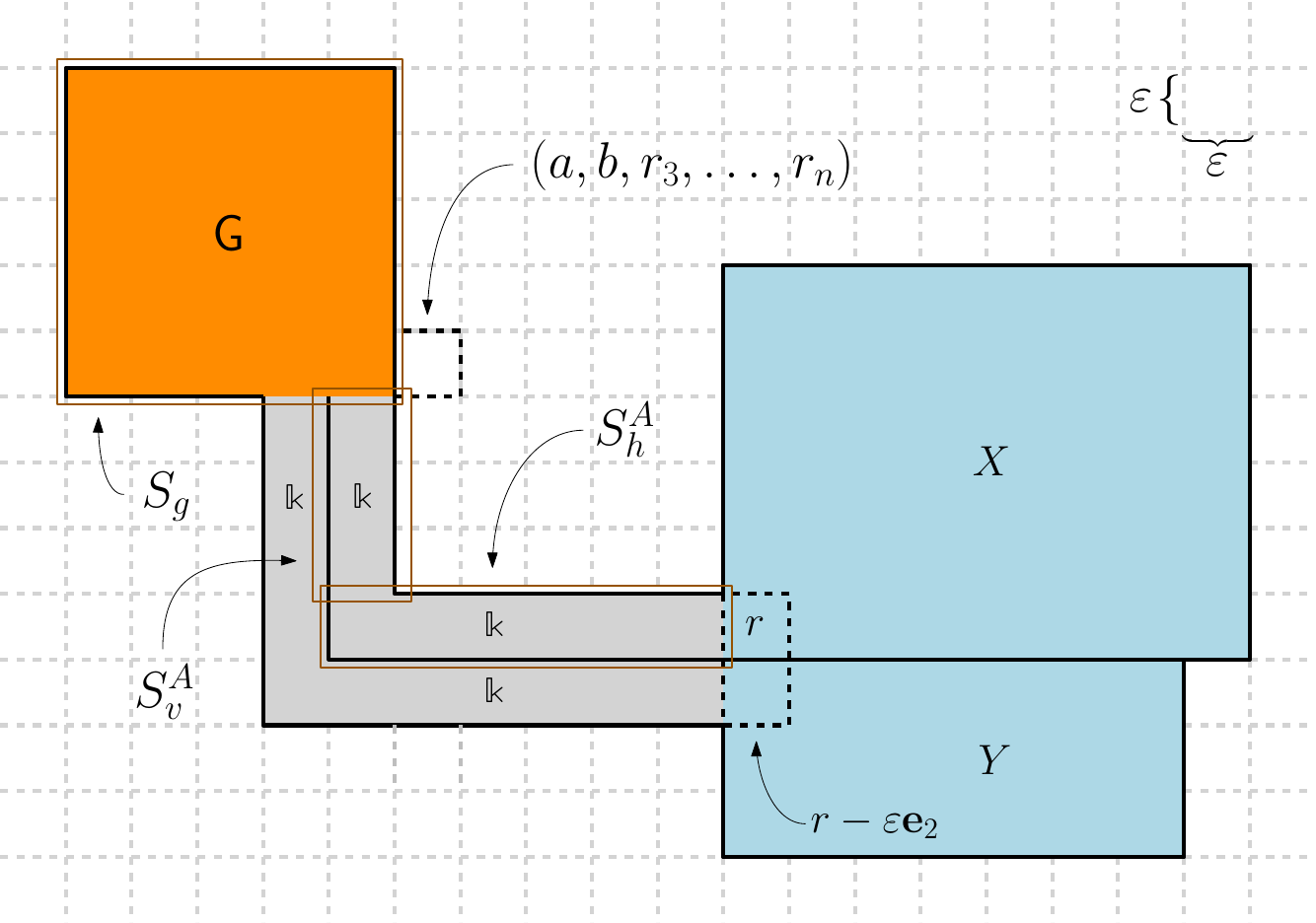}
    \caption{
    A schematic summary of the main constructions in the proof of \cref{lemma:actual-tacking}.
    Depicted is a $2$-dimensional slice corresponding to the first two coordinates, where coordinates $3$ to $n$ equal $r_3$ to $r_n$.}
    \label{figure:diagram-step-3}
\end{figure}

\begin{proof}
For a diagrammatic description of the main constructions in this proof, see \cref{figure:diagram-step-3}.

    Without loss of generality, we may assume that $\ell = 1$ and $\ell' = 2$.
    Choose $a,b \in \epsilon\Zbf$ such that $a < r_1$ and $b > r_2$, and such that $A(t) = B(t) = 0$ for all $t \in (\epsilon \Zbf)^n$ with $t_1 \leq a$.

    Define the following hyper-rectangles in $(\epsilon \Zbf)^n$:
    \begin{align*}
        S_g &= [a-5\epsilon,a-\epsilon] \times [b, b+4\epsilon] \times \left(\prod_{3 \leq k \leq n} \{r_k\}\right)\\
        S^A_h &= [a-\epsilon, r_1] \times \{r_2\}\times \left(\prod_{3 \leq k \leq n} \{r_k\}\right)\\
        S^B_h &= [a-2\epsilon, r_1] \times \{r_2-\epsilon\}\times \left(\prod_{3 \leq k \leq n} \{r_k\}\right)\\
        S^A_v &= \{a-\epsilon\} \times [r_2+\epsilon, b-\epsilon] \times \left(\prod_{3 \leq k \leq n} \{r_k\}\right)\\
        S^B_v &= \{a-2\epsilon\} \times [r_2, b-\epsilon]\times \left(\prod_{3 \leq k \leq n} \{r_k\}\right).
    \end{align*}
    We will modify $A \oplus B$ only on the $(5\epsilon)$-trivial set $\widehat{S_g} \cup \widehat{S^A_h} \cup \widehat{S^B_h} \cup \widehat{S^A_v} \cup \widehat{S^B_v}$, where hats denote extension to $\Rbf^n$, as in \cref{lemma:extension-set-is-trivial}.


    \medskip

    \noindent\emph{Discretize $A$ and $B$.}
    We let $X,Y : (\epsilon \Zbf)^n \to \vect$ denote the restriction of $A$ and $B$ to $(\epsilon \Zbf)^n$, respectively.
    \medskip

    \noindent\emph{Modify $X$ and $Y$.}
    With an argument analogous to that of the proof of \cref{lemma:moving-attachment}, we define an indecomposable module $X' : (\epsilon\Zbf)^n \to \vect$ (respectively $Y'$), which coincides with $X$ (respectively $Y$) everywhere, except on $S^A_h \cup S^A_v \cup \{r\}$ (respectively $S^B_h \cup S^B_v \cup \{r - \epsilon \esf_2$), where we set it to be constantly $\kbb$.

    \medskip

    \noindent\emph{Tack $X'$ and $Y'$ together with a copy of $\Gsf$.}
    Consider the module $Z = X' \oplus Y'$.
    We now define a module $W : \overline{S} \to \vect$.
    Let
    \[
        S_g' = [a-5\epsilon,a-\epsilon] \times [b-\epsilon, b+4\epsilon] \times \left(\prod_{3 \leq k \leq n} \{r_k\}\right).
    \]
    Note that $S_g \subseteq S_g' \subseteq \overline{S_g}$, define $W|_{\overline{S}\setminus S_g'} = 0$, and define $W|_{S_g'}$ using a copy of $\Gsf$ and two extra $\kbb$ in the bottom row as follows:
    \[ \footnotesize
        \begin{tikzpicture}
            \matrix (m) [matrix of math nodes,row sep=3em,column sep=3em,minimum width=2em,nodes={text height=1.75ex,text depth=0.25ex}]
            { \kbb & \kbb   & \kbb   & \kbb   & \kbb \\
              \kbb & \kbb^2 & \kbb^2 & \kbb^2 & \kbb \\
              0    & \kbb   & \kbb^2 & \kbb^2 & \kbb \\
              0    & 0      & \kbb   & \kbb^2 & \kbb \\
              0    & 0      & 0      & \kbb   & \kbb \\
              0    & 0      & 0      & \kbb      & \kbb    \\};
            \path[line width=0.5pt, -{>[width=6pt]}]
            (m-1-1) edge [-,double equal sign distance] (m-1-2)
            (m-1-2) edge [-,double equal sign distance] (m-1-3)
            (m-1-3) edge [-,double equal sign distance] (m-1-4)
            (m-1-4) edge [-,double equal sign distance] (m-1-5)

            (m-2-2) edge [-,double equal sign distance] (m-2-3)
            (m-2-3) edge [-,double equal sign distance] (m-2-4)

            (m-3-3) edge [-,double equal sign distance] (m-3-4)

            (m-5-5) edge [-,double equal sign distance] (m-4-5)
            (m-4-5) edge [-,double equal sign distance] (m-3-5)
            (m-3-5) edge [-,double equal sign distance] (m-2-5)
            (m-2-5) edge [-,double equal sign distance] (m-1-5)

            (m-4-4) edge [-,double equal sign distance] (m-3-4)
            (m-3-4) edge [-,double equal sign distance] (m-2-4)

            (m-3-3) edge [-,double equal sign distance] (m-2-3)

            (m-2-1) edge node [left] {\footnotesize $0$} (m-1-1)
            (m-2-1) edge node [above] {\footnotesize $\begin{pmatrix} 1\\ 1\end{pmatrix}$} (m-2-2)
            (m-2-2) edge node [right] {\footnotesize $(1,-1)$} (m-1-2)
            (m-2-3) edge node [right] {\footnotesize $(1,-1)$} (m-1-3)
            (m-2-4) edge node [right] {\footnotesize $(1,-1)$} (m-1-4)

            (m-5-4) edge node [above] {\footnotesize $0$} (m-5-5)
            (m-5-4) edge node [left] {\footnotesize $\begin{pmatrix} 1\\ 1\end{pmatrix}$} (m-4-4)
            (m-4-4) edge node [above] {\footnotesize $(1,-1)$} (m-4-5)
            (m-3-4) edge node [above] {\footnotesize $(1,-1)$} (m-3-5)
            (m-2-4) edge node [above] {\footnotesize $(1,-1)$} (m-2-5)

            (m-3-2) edge node [left] {\footnotesize $\begin{pmatrix} 1\\ 0\end{pmatrix}$} (m-2-2)
            (m-3-2) edge node [above] {\footnotesize $\begin{pmatrix} 1\\ 0\end{pmatrix}$} (m-3-3)

            (m-4-3) edge node [left] {\footnotesize $\begin{pmatrix} 0\\ 1\end{pmatrix}$} (m-3-3)
            (m-4-3) edge node [above] {\footnotesize $\begin{pmatrix} 0\\ 1\end{pmatrix}$} (m-4-4)

            (m-3-1) edge (m-2-1)
            (m-4-1) edge (m-3-1)
            (m-5-1) edge (m-4-1)

            (m-4-2) edge (m-3-2)
            (m-5-2) edge (m-4-2)

            (m-5-3) edge (m-4-3)

            (m-5-1) edge (m-5-2)
            (m-5-2) edge (m-5-3)
            (m-5-3) edge (m-5-4)

            (m-4-1) edge (m-4-2)
            (m-4-2) edge (m-4-3)
        
            (m-3-1) edge (m-3-2)

            (m-6-1) edge (m-6-2)
            (m-6-2) edge (m-6-3)
            (m-6-3) edge (m-6-4)
            (m-6-4) edge node [above] {\footnotesize $0$} (m-6-5)

            (m-6-1) edge (m-5-1)
            (m-6-2) edge (m-5-2)
            (m-6-3) edge (m-5-3)
            (m-6-4) edge node [left] {\footnotesize $1$} (m-5-4)
            (m-6-5) edge node [left] {\footnotesize $1$} (m-5-5)
            ;
        \end{tikzpicture}
    \]

    By \cref{corollary:modify-on-rectangle}(1) with $\delta = \epsilon$, hyper-rectangle $S = S_g$, $M = Z$, $N = W$, and $U = \overline{S}$, there exists a module $L$ such that $L|_{(\epsilon \Zbf)^n \setminus S_g} = Z|_{(\epsilon \Zbf)^n \setminus S_g}$ and $L|_{S_g} = W$.
    The module $L$ is also indecomposable by \cref{corollary:modify-on-rectangle}(2) since
    $N$ is indecomposable, and
    $L|_{(\epsilon \Zbf)^n \setminus S_g} \cong
        X'|_{(\epsilon \Zbf)^n \setminus S_g} \oplus
        Y'|_{(\epsilon \Zbf)^n \setminus S_g}$
    is a decomposition into indecomposables and both summands in the decomposition have non-trivial restrictions to $\overline{S}$.

    \medskip

    Finally, extend $L$ to the poset $\Rbf^n$ to get an indecomposable, finitely presentable module $M$, which coincides with $A \oplus B$ outside a $(5\epsilon)$-trivial set.
\end{proof}

We now recall the statement of the Tacking Lemma, before giving the proof.

\mainlemma*
\begin{proof}
    Let $\Pscr = (\tau\Zbf)^n$, for some $\tau> 0$.
    Let $m \in \Nbb$ be such that $\epsilon_0 \coloneqq \tau/m < \delta/4$.
    Note that $A$ and $B$ are also $(\epsilon_0\Zbf)^n$-extensions.

    Using \cref{lemma:add-1-dim-corner} and \cref{lemma:local-changes}, we replace $A$ and $B$ with $A_1$ and $B_1$ having a thin corner over $(\frac{\epsilon_0}{2}\Zbf)^n$, and at distance at most $\epsilon_0/2 < \delta/4$ from $A$ and $B$, respectively.

    Then, using \cref{lemma:add-antenna-attachment}  and \cref{lemma:local-changes}, we replace $A_1$ and $B_1$ by $A_2$ and $B_2$ having axis-$1$-aligned antenna attachments over
    $(\frac{\epsilon_0}{2\cdot 5}\Zbf)^n =
      (\frac{\epsilon_0}{10}\Zbf)^n$
    at $\alpha, \beta \in (\frac{\epsilon_0}{10}\Zbf)^n$, and at distance at most $\epsilon_0/2 < \delta/4$ from $A_2$ and $B_2$, respectively.
 
    Let $u \in (\frac{\epsilon_0}{10}\Zbf)^n$ be such that
    \begin{itemize}
        \item $u_k < \alpha_k - \epsilon_0/10$ and $u_k < \beta_k - \epsilon_0/10$ for all $1 \leq k \leq n$ odd;
        \item $u_k > \alpha_k + \epsilon_0/10$ and $u_k > \beta_k + \epsilon_0/10$ for all $1 \leq k \leq n$ even;
        \item $A_2(t) = B_2(t) = 0$ whenever $t_1 \leq u_1$.
    \end{itemize}
    Let $\ell = 1$ and $\ell' = 2$ if $n$ is even, and $\ell = n$ and $\ell' = n-1$ if $n$ is odd; and let $r = u + (\epsilon_0/10)\ebf_{\ell'}$. 
    We now use \cref{lemma:moving-attachment} and \cref{lemma:local-changes} to replace $A_2$ and $B_2$ by modules $A_3$ and $B_3$ having axis-$\ell$-aligned antenna attachments over $(\frac{\epsilon_0}{10}\Zbf)^n$ at $r$ and at $r-(\epsilon_0/10)\ebf_{\ell'}$ respectively, and at distance at most $\epsilon_0/10 < \delta/4$ from $A_2$ and $B_2$, respectively.

    Finally, we use \cref{lemma:actual-tacking} and 
    \cref{lemma:local-changes} to construct an indecomposable module $L$, which is an $((\epsilon_0/10)\Zbf)^n$-extension, and which is at distance at most $5(\epsilon_0/10) = (\epsilon_0/2) < \delta/4$ from $A_3 \oplus B_3$.

    Set $M = L$.
    To conclude, we have
    \begin{align*}
        d_I(M,A\oplus B) \leq\; & d_I(M,A_3 \oplus B_3) + d_I(A_3 \oplus B_3, A_2 \oplus B_2)\\
        &+ d_I(A_2 \oplus B_2, A_1 \oplus B_1) + d_I(A_1 \oplus B_1, A \oplus B) < \delta
    \end{align*}
    as claimed.
\end{proof}

%

\section{Nearly indecomposables are open}
\label{section:almost-indecomposable-open}

Before we can prove \cref{proposition:stability-indecomposability}, we need a definition and two lemmas.

\begin{dfn}
    The set of \emph{mesh widths} of a countable grid $\Pscr = \left(\{s_i\}_{i \in \Zbb}\right)^n \subseteq \Rbf^n$ is the set of differences of the grid coordinates $\{s_{i+1} - s_i : i \in \Zbb\}$.
\end{dfn}

Note that for any $\beta > 0$, every finite grid $S^n$ is included in some countable grid with mesh widths contained in the closed interval $[\alpha,\beta]$, for some $\alpha > 0$; one can take, for example, the countable grid $(S \cup (\beta\Zbf))^n$.

\begin{lem}
    \label{lemma:factor-bounded-grid}
    Let $0 < \alpha < \beta \in \Rbb$, let $\Pscr$ be a countable grid with mesh widths contained in the closed interval $[\alpha, \beta]$, and let $L : \Rbf^n \to \vect$.
    For every $r \leq \alpha$ there exists a morphism $m : L_{\Pscr + r}[r] \to L_\Pscr[\beta]$ rendering the following square commutative:
    \[
    \begin{tikzpicture}
        \matrix (m) [matrix of math nodes,row sep=4em,column sep=7em,minimum width=2em,nodes={text height=1.75ex,text depth=0.25ex}]
        {  L_{\Pscr}  &  L_{\Pscr}[\beta] \\
           L[r]_\Pscr & L_{\Pscr + r}[r],\\};
        \path[line width=0.75pt, -{>[width=8pt]}]
        (m-1-1) edge node [above] {$\eta^{L_\Pscr}_\beta$} (m-1-2)
        (m-1-1) edge node [left] {$(\eta^L_r)_\Pscr$} (m-2-1)
        (m-2-2) edge node [right] {$m$} (m-1-2)
        (m-2-1) edge [-,double equal sign distance] (m-2-2)
        ;
    \end{tikzpicture}
    \]
\end{lem}

\begin{proof}
    We start by defining the morphism $m$, which is equivalently a morphism $m : L_{\Pscr+r} \to L_\Pscr[\beta-r]$.
    If $x \in \Rbf^n$, then
    $L_{\Pscr + r}(x) = L(\sup\{ s + r : s \in \Pscr, s + r \leq x\})$ and
    \[
        L_{\Pscr}[\beta-r](x) = L_{\Pscr}(x + \beta - r) = L(\sup\{s : s \in \Pscr, s \leq x + \beta - r\}).
    \]
    Let $s_0$ such that $s_0 + r = \sup\{ s + r : s \in \Pscr, s + r \leq x\}$, so that we have $L_{\Pscr + r}(x) = L(s_0 + r)$.
    If $\Pscr = (\{s_i\}_{i \in \Zbb})^n$, then $s_0 = (s_{i_1}, \dots, s_{i_n})$.
    Let $s_1 = (s_{i_1+1}, \dots, s_{i_n+1})$.
    Note that, since $\Pscr$ has mesh widths contained in $[\alpha,\beta]$, we have $s_1 - \beta \leq s_0$.
    Thus, $s_1 - \beta + r \leq s_0 + r \leq x$, which implies $s_1 \leq x + \beta - r$.
    This means that
    \begin{equation}
        \label{equation:s1-inequality}
        s_1 \leq \sup\{s : s \in \Pscr, s \leq x + \beta - r\}.
    \end{equation}
    Also, $r \leq \alpha$, so we have
    \begin{equation}
        \label{equation:s0-inequality}
        s_0 + r \leq s_1
    \end{equation}
    We can then consider the composite:
    \begin{multline*}
        L_{\Pscr + r}(x) = L(s_0 + r)
        \xrightarrow{\eqref{equation:s0-inequality}}
        L(s_1)
        \xrightarrow{\eqref{equation:s1-inequality}}
        L(\sup\{s \in \Pscr : s \leq x + \beta - r\}) \\= L_{\Pscr}[\beta-r](x),
    \end{multline*}
    given by the structure morphisms of $L$.
    Let us denote the above composite by $m_x$.
    Since the morphism was constructed only using the structure morphisms of $L$, it is straightforward to check that the morphisms $m_x : L_{\Pscr + r}(x) \to L_{\Pscr}[\beta-r](x)$ assemble into a natural transformation $m : L_{\Pscr+r} \to L_\Pscr[\beta-r]$.
    By the same reason, it follows that $m$ renders commutative the square in the statement.
\end{proof}

\begin{lem}
    \label{lemma:snapping-trivial}
    Let $\epsilon , \beta > 0$, let $\Pscr$ be a countable grid with mesh widths bounded above by $\beta$, and let $L : \Rbf^n \to \vect$.
    If $L_\Pscr$ is (strictly) $\epsilon$-trivial, then $L$ is (strictly) $(\epsilon + \beta)$-trivial.
\end{lem}

\begin{proof}
We show the non-strict version of the statement, which directly implies the strict one.
Specifically, we prove that, for any $x \in \Rbf^n$, the morphism $L(x) \to L(x+(\epsilon + \beta))$ is zero.
Let $y \in \Pscr$ be minimal with
\begin{equation}
    \label{equation:inequality-xy}
    x \leq y.
\end{equation}
By the assumption on the mesh width, we have $y \leq x + \beta$ (where, as usual, the inequality is componentwise), and hence
\begin{equation}
    \label{equation:inquality-xy-epsilon-beta}
    y + \epsilon \leq x + (\epsilon + \beta).
\end{equation}
Let $z = \sup\{s \in \Pscr : s \leq y + \epsilon\}$.
Then, the assumption that $L_\Pscr$ is $\epsilon$-trivial implies that
the map $L(y) = L_\Pscr(y) \to L_\Pscr(y+\epsilon) = L(z)$ is zero.
Note that we must have
\begin{equation}
    \label{equation:inequality-yz}
    y \leq z \leq y+\epsilon,
\end{equation}
since $y \in \Pscr$.
We conclude that the composite map
\[
    L(x) \xrightarrow{\cref{equation:inequality-xy}}
    L(y) \xrightarrow{\cref{equation:inequality-yz}}
    L(z) \xrightarrow{\cref{equation:inequality-yz}}
    L(y+\epsilon) \xrightarrow{\cref{equation:inquality-xy-epsilon-beta}}
    L(x+(\epsilon + \beta))
\]
is zero, as required.
\end{proof}

\stabilityindecomposability*

We first give an outline of the proof.
    Let $M \cong M' \oplus C$, with $M'$ indecomposable and $C$ strictly $\epsilon$-trivial.
    We start by arguing that the persistence module $M$ is a $\Pscr$-extension for a sufficiently fine countable grid $\Pscr$.
    We then choose a sufficiently small $\delta > 0$ and consider an arbitrary module $N$ at interleaving distance at most $\delta$ from $M$.
    Next, we use \cref{lemma:structure-map-iso}, which says that the extension of a module over a grid can only change when one of its coordinates crosses a point in the grid; this is used to show that, after restricting to appropriate grids, a restriction of $M'$ is indecomposable and isomorphic to a direct summand of a restriction of $N$, up to a shift.
    This allows us to find a decomposition of $N$ into two summands, one of which is indecomposable and is related to a restriction of $M'$.
    Finally, we show that the other summand is strictly $\epsilon$-trivial.

\begin{proof}
    Let $M \cong M' \oplus C$ be finitely presentable, with $M'$ indecomposable and $C$ strictly $\epsilon$-trivial;
    we show that there exists $\delta > 0$ such that every persistence module $N : \Rbf^n \to \vect$ with $d_I(M,N) < \delta$ is $\epsilon$-indecomposable.

    By \cref{lemma:fp-is-extension-finite-poset}, the module $M$ is a $\Qscr$-extension for some finite grid $\Qscr = \{r_1 < r_2 < \dots < r_k\}^n \subseteq \Rbf^n$.
    By hypothesis, we have $M \cong M' \oplus C$, with $M'$ indecomposable and $C$ strictly $\epsilon$-trivial, that is, $C$ is $\epsilon'$-trivial for some $\epsilon' < \epsilon$.

    \medskip

    \noindent\emph{Fix a grid and choose $\delta$.}
    Choose
    \begin{equation}\label{eq:beta}
        \beta < \min(\epsilon - \epsilon', \epsilon'),
    \end{equation}
    and let $\Pscr = (\{s_i\}_{i \in \Zbb})^n$ be a countable grid containing $\Qscr$ and with mesh widths in $[\alpha,\beta]$ for some $\alpha >0$.
    Now choose $\delta < \alpha/2$.

    \medskip

    Assume given $N : \Rbf^n \to \vect$ with $d_I(M,N) < \delta$.
    
    \medskip

    \noindent\emph{Decompose $N$.}
    Given a $\delta$-interleaving $f : M \to N[\delta]$ and $g : N \to M[\delta]$ between $M$ and $N$, the interleaving equalities $g[\delta] \circ f = \eta^M_{2\delta}$ and $f[2\delta] \circ g[\delta] = \eta^{N[\delta]}_{2\delta}$ induce, by \cref{equation:shift-and-restriction-extension} (at the end of \cref{section:background}), the following commutative diagram of $\Rbf^n$-persistence modules:
    \begin{equation}
        \label{equation:decomposition-N}
        \begin{tikzpicture}[baseline=(current  bounding  box.center)]
            \matrix (m) [matrix of math nodes,row sep=4em,column sep=3em,minimum width=2em,nodes={text height=1.75ex,text depth=0.25ex}]
            {      & N[\delta]_{\Pscr}  & & N[3\delta]_{\Pscr} \\
                 M_{\Pscr} &  & M[2\delta]_{\Pscr} &  \\};
            \path[line width=0.75pt, -{>[width=8pt]}]
            (m-2-1) edge node [above] {$\cong$} node [below] {$(\eta^M_{2\delta})_{\Pscr}$} (m-2-3)
            (m-2-1) edge [>->] node [above left] {$f_{\Pscr}$} (m-1-2)
            (m-1-2) edge [->>] node [above right] {$g[\delta]_{\Pscr}$} (m-2-3)
            (m-1-2) edge node [above] {$(\eta^{N[\delta]}_{2\delta})_{\Pscr}$} (m-1-4)
            (m-2-3) edge node [below right] {$f[2\delta]_\Pscr$} (m-1-4)
            ;
        \end{tikzpicture}
    \end{equation}
    Note that $M$ is a $\Pscr$-extension and thus $M \cong M_\Pscr$.
    Thus, the bottom horizontal morphism is an isomorphism by \cref{lemma:structure-map-iso} and the fact that $2\delta < \alpha$.
    This implies that the left diagonal morphism is a split monomorphism, with left inverse $(\eta_{2\delta}^M)_\Pscr^{-1} \circ g[\delta]_\Pscr$.
    This induces an isomorphism $N[\delta]_\Pscr \cong M_\Pscr \oplus X$, for some module $X$, with the property that
    \begin{equation}
        \label{equation:inclusion-X-and-projection}
        \text{the composite}\;\; X \rightarrowtail M_\Pscr \oplus X \cong N[\delta]_\Pscr \xrightarrow{g[\delta]_\Pscr} M[2\delta]_\Pscr \;\; \text{is the zero morphism.}
    \end{equation}
    In particular, since $M_\Pscr \cong M \cong M' \oplus C$, we have $N[\delta]_{\Pscr} \cong M' \oplus C \oplus X$.
    
    At the same time, by \cref{theorem:decomposition}, there exists a decomposition $N \cong \bigoplus_{i \in I} N_i$ with $N_i$ indecomposable for all $i \in I$.
    Using \cref{lemma:decomposition-restriction-extension}, we see that, by restricting to $\Pscr$ and extending, we get a decomposition $N[\delta]_{\Pscr} \cong \bigoplus_{i \in I} (N_i[\delta])_{\Pscr}$, where now the summands $\{(N_i[\delta])_{\Pscr}\}_{i \in I}$ may not be indecomposable anymore.
    Nevertheless, since $N[\delta]_{\Pscr} \cong M' \oplus C \oplus X$ and $M'$ is indecomposable by assumption, by \cref{theorem:decomposition} there exists $i \in I$ such that $(N_i[\delta])_{\Pscr}\cong M' \oplus X'$ for some $X'$.
    We 
    let $A = N_i$ and $B = \bigoplus_{j \in I \setminus i} N_j$.
    Note also that $B[\delta]_\Pscr$ is a direct summand of $C \oplus X$.
    
    We have thus decomposed $N \cong A \oplus B$ with $A$ indecomposable.

    \medskip

    \noindent\emph{The module $B$ is strictly $\epsilon$-trivial.}
    To conclude, it remains to show that $B$ is strictly $\epsilon$-trivial, or, equivalently, that $B[\delta]$ is strictly $\epsilon$-trivial.
    \cref{lemma:snapping-trivial} reduces the problem to showing that $B[\delta]_\Pscr$ is strictly $(\epsilon-\beta)$-trivial.
    Since $B[\delta]_\Pscr$ is a direct summand of $C \oplus X$,
    it is enough to show that $C$ and $X$ are $(\epsilon-\beta)$-trivial.
    Note that, by \cref{eq:beta}, we have $\epsilon - \beta > \epsilon - \min(\epsilon - \epsilon', \epsilon') = \max(\epsilon', \epsilon - \epsilon')$, and hence $\epsilon',\beta < \epsilon - \beta$.
    The module $C$ is $\epsilon'$-trivial, by assumption, so the rest of this proof is devoted to showing that $X$ is $\beta$-trivial.

    We show that $\eta^X_{\beta} : X \to X[\beta]$ is the zero morphism.
    By the naturality of shifting, that is, by the commutativity of the following square
    \[
    \begin{tikzpicture}
        \matrix (m) [matrix of math nodes,row sep=4em,column sep=7em,minimum width=2em,nodes={text height=1.75ex,text depth=0.25ex}]
        {  M_\Pscr \oplus X  &  M_\Pscr[\beta] \oplus X[\beta] \\
           N[\delta]_{\Pscr}  & (N[\delta]_{\Pscr})[\beta],\\};
        \path[line width=0.75pt, -{>[width=8pt]}]
        (m-1-1) edge node [above] {$\left(\eta^{M_\Pscr}_\beta, \eta^X_\beta\right)$} (m-1-2)
        (m-1-1) edge node [left] {$\cong$} (m-2-1)
        (m-1-2) edge node [left] {$\cong$} (m-2-2)
        (m-2-1) edge node [above] {$\eta^{N[\delta]_{\Pscr}}_\beta$} (m-2-2)
        ;
    \end{tikzpicture}
    \]
    it is sufficient to show that the following composite is the zero morphism:
    \begin{equation}
        \label{equation:sufficient-zero-morphism}
        X \rightarrowtail M_\Pscr \oplus X \xrightarrow{\cong} N[\delta]_{\Pscr} \xrightarrow{\eta^{N[\delta]_{\Pscr}}_\beta} (N[\delta]_{\Pscr})[\beta].
    \end{equation}
    By \cref{equation:inclusion-X-and-projection}
    and the commutativity of the right interleaving triangle in \cref{equation:decomposition-N}, the following composite is the zero morphism
    \begin{equation}
        \label{equation:zero-morphism}
        X \rightarrowtail M_\Pscr \oplus X \xrightarrow{\cong} N[\delta]_{\Pscr}
        \xrightarrow{(\eta^{N[\delta]}_{2\delta})_{\Pscr}} N[3\delta]_{\Pscr}.
    \end{equation}
    Thus, it suffices to see that the right-most morphism in \cref{equation:sufficient-zero-morphism} factors through the right-most morphism in \cref{equation:zero-morphism}.
    Letting $L = N[\delta]$ 
	and $r = 2\delta < \beta$, 
    this translates to the claim that the morphism 
    $\eta^{L_{\Pscr}}_\beta : L_{\Pscr} \to L_{\Pscr}[\beta]$ factors through $(\eta^{L}_{r})_{\Pscr} : L_{\Pscr} \to L[r]_{\Pscr}$, which follows from \cref{lemma:factor-bounded-grid}.
%
\end{proof}

\section{Consequences}
\label{section:consequences}

In this section we prove \cref{theorem:main-theorem,corollary:structural-stability,theorem:dense-F-decomposables,theorem:dense-F-complicated}.

\maintheorem*
\begin{proof}
    Let $\epsilon > 0$.
    The set of $\epsilon$-indecomposable $n$-parameter persistence modules is open, by \cref{proposition:stability-indecomposability}.
    
    This set also contains all indecomposable, finitely presentable modules, by definition.
    Since the set of indecomposables is dense by \cref{theorem:indecomposables-dense}, the result follows.
\end{proof}

Before proving \cref{corollary:structural-stability}, let us recall the notions of matching, bottleneck distance, and structural stability given in the introduction

Let $\epsilon > 0 \in \Rbb$, and let $M, N : \Rbf^n \to \vect$.
An \emph{$\epsilon$-matching} between $M$ and $N$ consists of two decompositions $M \cong \bigoplus_{j \in J} M_j$ and $N \cong \bigoplus_{j \in J} N_j$, indexed by the same set $J$, where each summand is either indecomposable or zero,
and such that $M_j$ and $N_j$ are $\epsilon$-interleaved for every $j \in J$.
The \emph{bottleneck distance} between $M$ and $N$ is
\[
    d_B(M,N) = \inf \{\epsilon \geq 0 : \text{there exists an $\epsilon$-matching between $M$ and $N$}\}.
\]
A finitely presentable module $M : \Rbf^n \to \vect$ is \emph{structurally stable} if, for every $\epsilon > 0$, there exists $\delta > 0$ such that every finitely presentable module $N : \Rbf^n \to \vect$ with $d_I(M,N) < \delta$ satisfies $d_B(M,N) < \epsilon$.

\structuralstability*
\begin{proof}
    Assume that $M$ is finitely presentable and structurally stable.
    Let $M$ decompose into indecomposables as $M \cong \bigoplus_{i \in I} M_i$.
    Since $M$ is finitely presentable, $I$ is finite, by \cref{lemma:decomposition-fp}, so let $\epsilon > 0$ be such that $d_I(M_i,0) > \epsilon$ for all $i \in I$.
    Since $M$ is structurally stable, there exists $\delta > 0$ such that, for every finitely presentable $N : \Rbf^n \to \vect$ with $d_I(M,N) < \delta$, there exists an $\epsilon$-matching between $M$ and $N$.
    By \cref{theorem:indecomposables-dense}, there exists an indecomposable, finitely presentable $N$ with $d_I(M,N) < \delta$.
    Since all of the indecomposable summands of $M$ are at distance more than $\epsilon$ from the zero module,
    in an $\epsilon$-matching between $M$ and $N$ given by decompositions $M \cong \bigoplus_{j \in J} M_j$ and $N \cong \bigoplus_{j \in J} N_j$, we must have that $M_j \neq 0$ implies $N_j \neq 0$.
    Since $N$ is indecomposable, there is at most one $j \in J$ such that $M_j$ is non-zero, and thus $M$ is indecomposable or zero.
    
    \medskip

    Assume now that $M$ is finitely presentable, and let $\epsilon > 0$.
    We prove that, if $M$ is zero or indecomposable, then there exists $\delta > 0$ such that, for every finitely presentable $N : \Rbf^n \to \vect$ with $d_I(M,N) < \delta$, we have $d_B(M,N) < \epsilon$.

    If $M = 0$, then we can set $\delta = \epsilon$, since $d_I(0,N) = d_I(M,N) < \epsilon$ implies that $d_I(0,N') < \epsilon$ for every indecomposable summand $N'$ of $N$, so $M \cong \bigoplus_{j \in J} M_j$ and $N \cong \bigoplus_{j \in J} N_j$, with $M_j = 0$ for all $j \in J$ and $N \cong \bigoplus_{j \in J} N_j$ a decomposition into indecomposables gives an $\epsilon$-matching between $M$ and $N$.

    Assume now that $M$ is indecomposable.
    By \cref{proposition:stability-indecomposability}, there exists $\delta > 0$ such that every persistence module $N : \Rbf^n \to \vect$ with $d_I(M,N) < \delta$ decomposes as $N \cong A \oplus B$ with $A$ indecomposable and $B$ strictly $\epsilon$-trivial.
    Let $B \cong \bigoplus_{1 \leq j \leq k} B_j$ be a decomposition into indecomposables.
    Then, the matching between $M$ and $N$ given by the decompositions
    $M \cong \bigoplus_{0 \leq j \leq k} M_j$ with $M_0 = M$ and $M_j = 0$ for every $1 \leq j \leq k$, and $N \cong \bigoplus_{0 \leq j \leq k} N_j$ with $N_0 = A$ and $N_j = B_j$ for every $1 \leq j \leq k$ is an $(\frac\epsilon2)$-matching.
    Thus, if $d_I(M,N) < \delta$, then $d_B(M,N) < \frac\epsilon2 \leq \epsilon$, as required.
\end{proof}

We use the following result in the proof of \cref{theorem:dense-F-decomposables}.

\begin{lem}
    \label{lemma:indecomposable-approximated-by-indecomposable}
    Let $\Fcal$ be a collection of isomorphism classes of indecomposable, finitely presentable $n$-parameter persistence modules.
    If an indecomposable, finitely presentable $n$-parameter persistence module is in the topological closure of the $\Fcal$-decomposable modules, then the module is in the closure of $\Fcal$ itself.
\end{lem}
\begin{proof}
    Let $M : \Rbf^n \to \vect$ be indecomposable, finitely presentable, and in the closure of the $\Fcal$-decomposable modules.
    Given $\epsilon > 0$, we show that there exists $N \in \Fcal$ such that $d_I(M,N) < \epsilon$.

    By \cref{proposition:stability-indecomposability}, there exists $\delta > 0$ such that every module at distance less than $\delta$ from $M$ is $\epsilon$-indecomposable.
    Let $\delta' = \min(\delta,\frac\epsilon2)$.
    Since $M$ is in the closure of the $\Fcal$-decomposable modules,
    there exists an $\Fcal$-decomposable module $X$ such that $d_I(M,X) < \delta'$.
    It follows that $X \cong N \oplus B$ with $N$ indecomposable and $B$ strictly $\epsilon$-trivial.
    In particular $N \in \Fcal$ and $d_I(N,X) = d_I(N,N \oplus B) < \frac\epsilon2$.
    To conclude, note that $d_I(M,N) \leq d_I(M,X) + d_I(X,N) < \epsilon$.
\end{proof}

\denseFdecomposables*
\begin{proof}
    Let $\Fcal$ be as in the statement, and let $\Gcal$ denote the set of finitely presentable, $\Fcal$-decomposable modules.
    Since $\Fcal \subseteq \Gcal$, we have $(\overline{\Fcal})^\circ \subseteq (\overline{\Gcal})^\circ$.
    To prove the other inclusion, it is enough to prove that $(\overline{\Gcal})^\circ \subseteq \overline{\Fcal}$.
    So let $M \in (\overline{\Gcal})^\circ$ and $\epsilon > 0$.
    We prove that there exists $N \in \Fcal$ with $d_I(M,N) < \epsilon$.

    Let $0 < \epsilon' \leq \epsilon$ be such that the open $\epsilon'$-ball around $M$ is included in $(\overline{\Gcal})^\circ$.
    By \cref{theorem:indecomposables-dense}, there exists an indecomposable, finitely presentable $M' : \Rbf^n \to \vect$ such that $d_I(M,M') < \epsilon'/2$.
    In particular, $M' \in (\overline{\Gcal})^\circ \subseteq \overline{\Gcal}$, so, by \cref{lemma:indecomposable-approximated-by-indecomposable}, there exists $N \in \Fcal$ such that $d_I(M',N) < \epsilon'/2$.
    It follows that $d_I(M,N) < \epsilon' \leq \epsilon$, as required.
\end{proof}

The following definition will be used in the proofs of
\cref{theorem:dense-F-complicated,lemma:not-baire}.
Given $a < b \in \Rbf^n$, let $\kbb_{[a,b)} : \Rbf^n \to \vect$ be the persistence module that takes the value $\kbb$ in the set $[a,b) = \{j \in \Rbf^n : a \leq j \text{ and } j \not\geq b\}$ and the value $0$ otherwise, with all structure morphisms that are not constrained to be zero being the identity.

\denseFcomplicated*
\begin{proof}
    We prove that, if the topological closure of the $\Fcal$-decomposable modules has non-empty interior, then for every $i \in \Rbf^n$ and $k \in \Nbb$, there exists $N \in \Fcal$ such that $\dim(N(i)) \geq k$.

    Let $M : \Rbf^n \to \vect$ be a finitely presentable module in the interior of the closure of the $\Fcal$-decomposable modules.
    By definition of the topology induced by a metric, there exists $\epsilon > 0$ such that the open $\epsilon$-ball around $M$ is contained in the closure of the $\Fcal$-decomposable modules.
    Let
    \[
        M' \coloneqq M \oplus \kbb_{[i-\epsilon,i+\epsilon)}^k
    \]
    be the direct sum of $M$ and $k$ copies of $\kbb_{[i-\epsilon,i+\epsilon)}$.
    Since $\kbb_{[i-\epsilon,i+\epsilon)}$ is $2\epsilon$-trivial, the modules $M$ and $M'$ are $\epsilon$-interleaved, so $M'$ is in the interior of the closure of the $\Fcal$-decomposable modules.
    Note that, by construction, $\rk(M'(i-\frac\epsilon2) \to M(i+\frac\epsilon2)) \geq k$.

    Using \cref{theorem:dense-F-decomposables}, let $N \in \Fcal$ be such that $d_I(M',N) < \frac\epsilon4$.
    We have $\rk(N(i-\frac{3\epsilon}4) \to N(i+\frac{3\epsilon}4)) \geq \rk(M'(i-\frac\epsilon2) \to M'(i+\frac\epsilon2)) = k$, since 
    the morphism $N(i-\frac{3\epsilon}4) \to N(i+\frac{3\epsilon}4)$ factors through $M'(i-\frac\epsilon2) \to M'(i+\frac\epsilon2)$, using any $\frac\epsilon4$-interleaving between $N$ and $M'$.
    In particular, we must have $\dim(N(i)) \geq k$, as required.
\end{proof}

\appendix

\section{Proofs of standard results}
\label{appendix}

   In this section, we provide proofs for standard results used in the main part of the paper.

\begin{proposition*}
\statementmultiparameterwild
\end{proposition*}
\begin{proof}
    We start with an observation: if $\Tscr$ is a poset, and $\Sscr \subseteq \Tscr$
    is a full subposet, we get, by left Kan extension, an additive, fully faithful functor $\vect^\Sscr \to \vect^\Tscr$, which, by uniqueness of decompositions, induces an injective function mapping isomorphism classes of indecomposable, finitely presentable $\Sscr$-persistence modules to isomorphism classes of indecomposable, finitely presentable $\Tscr$-persistence modules.
    Consider now the subposet $\Qscr$ of $\{0,1,2,3\} \times \{0,1,2\}$, represented as by its Hasse diagram as follows:
    \[\footnotesize
        \begin{tikzpicture}
            \matrix (m) [matrix of math nodes,row sep=3em,column sep=3em,minimum width=2em,nodes={text height=1.75ex,text depth=0.25ex}]
            {   & (1,2) &    &    \\
               (0,1) & (1,1) & (2,1) & (3,1) \\
                   & (1,0)      &       &    \\};
            \path[line width=0.5pt, -{>[width=6pt]}]
            (m-2-1) edge (m-2-2)
            (m-2-2) edge (m-2-3)
            (m-2-3) edge (m-2-4)
            (m-3-2) edge (m-2-2)
            (m-2-2) edge (m-1-2)
            ;
        \end{tikzpicture}
    \]
    If $\Pscr$ is a totally ordered set with at least $4$ elements, and $n \geq 2$, then $\Qscr$ embeds in $\Pscr^n$.
    It follows that there exists an explicit injective function mapping isomorphism classes of indecomposable, finitely presentable $\Qscr$-persistence modules to isomorphism classes of indecomposable, finitely presentable $\Pscr^n$-persistence modules.
    The function is explicit in the sense that, if one is given a $\Qscr$-persistence module described by means of a basis for each vector space and a matrix for each structure morphism, one can explicitly write down bases and matrices for the corresponding $\Pscr^n$-persistence module.

    Now, the well known fact that $\Qscr$, which is in fact a quiver, is of wild representation type finishes the proof.
    In more detail: by \cite[Chapter~14.2,~Corollary~14.11]{simson} and \cite[Lemma~4~(3)]{nazarova-2}, there exists a functor $\mod(\Lambda) \to \vect^\Qscr$ inducing an injective function mapping isomorphism classes of indecomposable, finite dimensional $\Lambda$-modules to isomorphism classes of indecomposable, finitely presentable $\Qscr$-persistence modules.
    This function is also explicit: see the construction in the proof of \cite[Chapter~14.2,~Proposition~14.10]{simson}, which gives an explicit formula that only requires a set of generators of $\Lambda$ as a $\kbb$-vector space.

    The result then follows by composing the two functions $\mod(\Lambda) \to \vect^\Qscr \to \vect^{\Pscr^n}$.
\end{proof}

   We now prove a lemma relating local rings to the definitions used in~\cite{azumaya}, which we now recall.
   The result is mentioned in \cite{azumaya}, but we include a proof here for completeness.

   A nonzero idempotent element $x \in R$ of a ring $R$ is \define{primitive} if every time we have $x = i + e$ with $i$ and~$e$ orthogonal idempotents, then $i = 0$ or $e = 0$.
   An element $c \in R$ is a \define{root element} if $cR$ contains no non-zero idempotent elements.
   If the set $C$ of all root elements of $R$ is closed under addition, then we say that $R$ \define{possesses the radical $C$}.
   The ring $R$ is \define{completely primary} if it possesses a radical, and every non-zero idempotent element of $R$ is primitive.
   Finally, a ring is \define{local} if the sum of any two non-invertible elements is non-invertible.

\begin{lem}
    \label{lemma:local-implies-completely-primary}
    Let $R$ be a unital ring with $1 \neq 0 \in R$.
    Then $R$ is local if and only it is completely primary.
\end{lem}
\begin{proof}
     We use the following well known facts about local rings.
     For a unital ring the following are equivalent: the ring is local; the ring admits a unique maximal left ideal; the ring admits a unique maximal right ideal.
     The unique maximal ideal of a local ring consists of all non-invertible elements.
     In a local ring the only two idempotent elements are $0$ and~$1$.
     For an element of a local ring the following are equivalent: the element is non-invertible; the element does not have a left inverse; the element does not have a right inverse.

    Let $R$ be local.
    We start by showing that $R$ possesses a radical $C$.
    We claim that $C$ is equal to the unique maximal ideal of non-invertible elements.
    To see this, let $c \in R$.
    We have that $cR$ contains no non-zero idempotents if and only if $1 \notin cR$, which happens if and only if $c$ is non-invertible.

    To conclude, we need to show that $1$ is primitive, which follows directly from the fact that the only two idempotents are $0$ and $1$.

    \medskip
    
    Now let $R$ be completely primary. We claim that the radical $C$ is the set of not right-invertible elements,
    that this set forms a right ideal, and that every proper right ideal of $R$ consists of not right-invertible elements.
    
    Let $c$ be a root element. Then $cR$ does not contain the idempotent $1$, so $c$ has no right inverse.
    Conversely, assume that $c$ has no right inverse.
    Assume that $cd \neq 1$ is an idempotent.
    Then $1-cd \neq 0$ is another idempotent, orthogonal to $cd$.
    Since $1 = cd + (1-cd)$ is primitive, we must have $cd=0$.
    Thus $c$ is a root element.
    
    Now if $c$ is not right-invertible, then $cd$ is not right-invertible, as $cde=1$ would show that $de$ is a right inverse to $c$. Thus $CR = C$. Furthermore, $C$ is closed under addition, and thus it is a right ideal.
    
    Finally, let $I$ be a proper right ideal of $R$. 
    Then $1 \notin I$.
    Let $c \in I$ and $d \in R$. Then $cd \in I$ and hence $cd \neq 1$.
    Thus $c$ has no right-inverse.    
\end{proof}

\theoremdecomposition*
\begin{proof}
    The existence of the decomposition as a direct sum of indecomposable follows directly from \cite[Theorem~1.1]{botnan-crawleybovey}.
    For the uniqueness of the decomposition, we use \cite[Theorem~1,~(ii)]{azumaya}, and, to satisfy the hypothesis of this theorem, we use the fact that indecomposable persistence modules have local endomorphism ring by \cite[Theorem~1.1]{botnan-crawleybovey}, and that local rings are completely primary (\cref{lemma:local-implies-completely-primary}) in the sense of \cite{azumaya}.
    Note that \cite{azumaya} is working with modules over an arbitrary ring; the results apply nevertheless, since the category of $\Pscr$-persistence modules admits an additive, fully faithful embedding into a category of modules (see, e.g., \cite[Lemma~2.1]{botnan-oppermann-oudot-scoccola}).
\end{proof}


%
%
%
%

\decompositionfp*
\begin{proof}
    By \cref{lemma:fp-is-extension-finite-poset}, there exists a finite grid $\Pscr \subseteq \Rbf^n$ such that $M$ is isomorphic to the extension of some $N : \Pscr \to \vect$.
    By \cref{theorem:decomposition}, there exists a family $\{N_i\}_{i \in I}$ with $N_i : \Pscr \to \vect$ such that $N \cong \bigoplus_{i \in I} N_i$ and $N_i \neq 0$.
    The set $I$ must be finite, since the poset~$\Pscr$ is finite and each module $N_i$ takes a non-zero value in at least one element of the poset~$\Pscr$.
    Thus, we have that $M \cong \bigoplus_{i \in I} \widehat{N_i}$, with $I$ finite.
    To conclude, use that \cref{lemma:fp-is-extension-finite-poset}, which implies that $\widehat{N_i}$ is finitely presentable for all $i \in I$.
\end{proof}

Finally, the following lemma does not seem to appear in the literature but is straightforward to prove using an interleaving argument. Similar results are known for the one-parameter case; see, e.g., \cite{bubenik-spaces-persistence}.

\begin{lem}
    \label{lemma:not-baire}
    Let $n \geq 1$.
    The space of finitely presentable, $n$-parameter persistence modules metrized with the interleaving distance is not Baire.
\end{lem}
\begin{proof}
    Given a finitely presentable module $M : \Rbf^n \to \vect$, let $c(M) = \max_{i \in \Rbf^n} \dim(M(i)) \in \Nbb$.
    We claim that, for every $C \geq 0$, the set $S_C \coloneqq \{M : c(M) \geq C \}$ is open and dense; the result follows from this claim since $\bigcap_{C \in \Nbb} S_C = \emptyset$.

    To prove that $S_C$ is dense, note that, given any finitely presentable module $M$ and $\epsilon > 0$, we can consider $M \oplus B$ with
    $B \coloneqq \kbb_{[0,\epsilon)}^C$.
    We have $c(M\oplus B) \geq c(B) = C$ and $d_I(M,M\oplus B) \leq \frac\epsilon2$, so $S_C$ is dense.

    To see that $S_C$ is open, let $M \in S_C$ and let $i \in \Rbf^n$ such that $M(i) \geq C$.
    Since $M$ is finitely presentable, there must exist $\epsilon > 0$ such that $\rk(M(i) \to M(i+2\epsilon)) \geq C$, by \cref{lemma:structure-map-iso}.
    For convenience, let $j = i+\epsilon$.
    Let $N : \Rbf^n \to \vect$ be such that $d_I(M,N) < \frac\epsilon4$; we prove that $N \in S_C$.
    We have $\rk(N(j-\frac{3\epsilon}4) \to N(j+\frac{3\epsilon}4)) \geq \rk(M(j-\frac\epsilon2) \to M(j+\frac\epsilon2)) = C$, since the morphism $N(j-\frac{3\epsilon}4) \to N(j+\frac{3\epsilon}4)$ factors through $M(j-\frac\epsilon2) \to M(j+\frac\epsilon2)$, using any $\frac\epsilon4$-interleaving between $N$ and $M$.
    In particular, we must have $\dim(N(j)) \geq C$, and thus $N \in S_C$, as required.

%
\end{proof}

\printbibliography

\end{document}